\documentclass[reqno]{amsart}
\usepackage{amsmath,amsfonts,amssymb,amsthm}
\usepackage[usenames,dvipsnames]{xcolor}
\usepackage[colorlinks=true]{hyperref} %para PFDLaTeX
\usepackage{hyperref}
\hypersetup{linkcolor=Violet,citecolor=PineGreen}
\newtheorem{teor}{Theorem}[section]
\newtheorem{lema}[teor]{Lemma}
\newtheorem{prop}[teor]{Proposition}

\theoremstyle{definition}
\newtheorem{defi}[teor]{Definition}

\newtheorem{nota}[teor]{Remark}

\numberwithin{equation}{section}
\newcommand\JM{Mierczy\'nski}

\newcommand{\N}{\mathbb{N}}
\newcommand{\R}{\mathbb{R}}
\newcommand{\T}{\mathbb{T}}
\newcommand{\Z}{\mathbb{Z}}

\newcommand{\PP}{\mathbb{P}}
\newcommand{\be}{\mathbf e}
\newcommand{\w}{\omega}

\newcommand{\nbd}{\nobreakdash}
\newcommand{\n}[1]{\|#1\|}

\newcommand{\Int}{\displaystyle\int }

\newcommand{\lsm}{\left[\begin{smallmatrix}}
\newcommand{\rsm}{\end{smallmatrix}\right]}
\newcommand{\mlsps}{measurable linear skew\nobreakdash-\hspace{0pt}product
semidynamical system}
\DeclareMathOperator{\spanned}{span}
\newcommand{\OFP}{\ensuremath{(\Omega,\mathfrak{F},\PP)}}

\DeclareMathOperator{\lnplus}{ln^{+}}
\begin{document}
\title[Principal Floquet subspaces and exponential separations of type II]
{Principal Floquet subspaces and exponential separations of type II
with applications to random delay differential equations.}
\author[J.~Mierczy\'nski]{Janusz Mierczy\'nski}
\address[Janusz Mierczy\'nski]{Faculty of Pure and Applied Mathematics,
Wroc{\l}aw University of Science and Technology,
Wybrze\.ze Wyspia\'nskiego 27, PL-50-370 Wroc{\l}aw, Poland.}
\email[Janusz Mierczy\'nski]{Janusz.Mierczynski@pwr.edu.pl}
\address[Sylvia Novo, Rafael Obaya]
{Departamento de Matem\'{a}tica Aplicada, Universidad de
Valladolid, Paseo del Cauce 59, 47011 Valladolid, Spain.}
\email[Sylvia Novo]{sylnov@wmatem.eis.uva.es}
\email[Rafael Obaya]{rafoba@wmatem.eis.uva.es}
\thanks{The first author is supported by the NCN grant Maestro 2013/08/A/ST1/00275 and the last two authors are partly supported by MEC (Spain)
under project MTM2015-66330-P and EU Marie-Sk\l odowska-Curie ITN Critical Transitions in
Complex Systems (H2020-MSCA-ITN-2014 643073 CRITICS)}
\author[S.~Novo]{Sylvia Novo}
\author[R.~Obaya]{Rafael Obaya}
\date{}
\begin{abstract}
This paper deals with the study of principal Lyapunov exponents, principal Floquet subspaces, and exponential separation for positive random linear dynamical systems in ordered Banach spaces. The main contribution lies in the introduction of a new type of exponential separation, called of type II, important for its application to  nonautonomous random differential equations with delay. Under weakened assumptions, the existence of an exponential separation of type II in an abstract general setting is shown, and an illustration of its application to dynamical systems generated by scalar linear random delay differential equations with finite delay is given.
\end{abstract}
\maketitle
%%%%%%%%%%%%%%%%%%%%%%%%%%%%%%%%%%%%%%%%%%%%%%%%%%%%%%%%%%%%%%%%%%%%%%%%%%%%%%%%%%%%%
%%%%%%%%%%%%%%%%%%%%%%%%%%%%%%%%%%%%%%%%%%%%%%%%%%%%%%%%%%%%%%%%%%%%%%%%%%%%%%%%%%%%%
%%%%%%%%%%%%%%%%%%%%%%%%%%%%%%%%%%%%%%%%%%%%%%%%%%%%%%%%%%%%%%%%%%%%%%%%%%%%%%%%%%%%%
%%%%%%%%%%%%%%%%%%%%%%%%%%%%%%%%%%%%%%%%%%%%%%%%%%%%%%%%%%%%%%%%%%%%%%%%%%%%%%%%%%%%%
\section{Introduction}\label{sec-intro}
This paper continues the study of the existence of principal  Lyapunov exponents, principal Floquet subspaces and  generalized exponential separations  for positive random linear skew-product semiflows in ordered Banach spaces. In particular, the concept of generalized exponential separation of type II is introduced as a natural modification of the classical concept, to later show the applicability of  this new theory in the context of nonautonomous functional differential equations with finite delay.
\par
\smallskip
Lyapunov exponents play an important role in the study of  deterministic and random skew-product semiflows. Oseledets~\cite{osel} obtained important results on Lyapunov exponents  and measurable invariant families of subspaces for finite-dimensional linear dynamical systems, currently referred  to  as  Oseledets multiplicative ergodic teorem. An important number of alternative proofs of this theorem as well as extensions of the theory to relevant infinite dimensional dynamical systems have been provided (see Arnold~\cite{Arn}, Gonz\'{a}lez-Tokman and Quas~\cite{GTQu}, Johnson \emph{et al.}~\cite{JoPaSe}, Krengel~\cite{Kre}, Lian and Lu~\cite{Lian-Lu}, Ma\~{n}\'{e}~\cite{Man}, Million\v{s}\v{c}ikov~\cite{Mil}, Raghunathan~\cite{Rag}, Ruelle~\cite{Rue1,Rue2}, Thieullen~\cite{Thieu}, and the references therein).
\par
\smallskip
The largest finite Lyapunov exponent (or top Lyapunov exponent) and its associated invariant subspace play a relevant role in this theory. Classically, the top finite Lyapunov exponent of a positive deterministic or random dynamical system in an ordered Banach space is called the \emph{principal Lyapunov exponent} if the associated invariant family of subspaces, where this Lyapunov exponent is reached, is one-dimensional and spanned by a positive vector. In this case the invariant subspace is called the \emph{principal Floquet subspace}.
\par
\smallskip
The \emph{exponential separation} theory was initiated for positive discrete-time deterministic dynamical systems by Ruelle~\cite{Rue0}, and developed later by Pol\'{a}\v{c}ik and Tere\v{s}\v{c}\'{a}k~\cite{pote,pote1}. Given a strongly ordered Banach space  $X$ and $\theta\colon \Omega \rightarrow \Omega $ a homeomorphism of a compact set $\Omega $, if  $\Omega  \times X \to \Omega \times X$, $(\omega,u) \mapsto (\theta(\omega),T_\omega u)$ defines a vector bundle map with $T_\omega$ compact and strongly positive for each $\omega \in \Omega$, then it admits a continuous decomposition $X=E(\omega) \oplus F(\omega)$ where $E(\omega)$ is the principal Floquet subspace, $F(\omega)$ does not contain any strictly positive vector and the bundle map exhibits an exponential separation on the sum. This statement was generalized by Shen and Yi~\cite{shyi} to continuous-time deterministic (topological) skew-product semiflows $ \R^+ \times \Omega \times X \mapsto \Omega \times X, (t,\omega,u) \rightarrow (\theta_t\w,U_\w(t)\,u)$, where $U_\w(t)$ is strongly positive for each $\omega \in \Omega$ and $t > 0$.   In a typical instance of application, $\Omega$ is the translation hull of the coefficients of some linear differential equation, that is, the closure, in an appropriate topology, of the set of all time-translates of the coefficients, and $U_\w(t)$ is the solution operator of the equation.  Applications and extensions of this theory in the context of nonautonomous ordinary and parabolic partial differential equations can be found in H\'uska and Pol\'{a}\v{c}ik~\cite{HusPol},  H\'uska~\cite{Hus}, H\'uska~\emph{et al.}~\cite{HusPolSaf}, and Novo \emph{et al.}~\cite{NOS2,noos5}, among other references.  The theory of principal Floquet spaces and exponential separations was developed further by \JM\ and Shen~\cite{MiSh1,MiSh3}.  Among others, they consider random families of parabolic linear partial differential equations whose coefficients are evaluated along the trajectories of a measurable dynamical system on a probability space: a random family is of such a type that it is embedded into a continuous deterministic (nonautonomous) family of linear
equations which in its turn generates a (topological) skew-product flow exhibiting an exponential separation.
\par
\smallskip
Novo \emph{et al.}~\cite{NOS2} introduced a  modification of the notion of exponential separation. A linear skew-product semiflow $ \R^+ \times \Omega \times X \rightarrow \Omega \times X, (t,\omega,u) \rightarrow (\theta_t\w,U_\w(t)\,u)$ is considered and  the strong monotonicity condition is substituted by the following dichotomy behaviour:  there exist times $0<t_1<T$ such that $U_\w(t)$ is a compact operator for $t \geq T-t_1$ and $\w\in\Omega$, and for each vector $u \geq 0$ either $U_\w(t_1)\,u=0$ or $U_\w(t_1)\,u\gg 0$.  Under these assumptions,  the existence of a continuous decomposition $X=E(\omega) \oplus F(\omega)$  is proved, where $E(\omega)$ is the principal Floquet subspace and the semiflow  exhibits an exponential separation on the sum, but now $F(\omega) \cap X^+$  is not void and contains those positive vectors $u$ satisfying $U_\w(t_1)\,u=0$. This dynamical behaviour is called exponential separation  (or continuous separation) of type~II, implicity referring to the classical concept as  exponential separation of type~I. Novo \emph{et al.}~\cite{noos5}, Calzada \emph{et al.}~\cite{COS} and Obaya and Sanz~\cite{OS1,OS2} show the importance of the exponential separation of type II in the study of linear and nonlinear nonautonomous functional differential equations with finite delay.
\par
\smallskip
In the case where the coefficients of the linear differential equation are driven by trajectories of a measurable flow $\theta$ on a probability space $\OFP$, the natural setting is that of positive \textit{measurable} linear skew-product semiflows:  $U_{\omega}(t)$ is a positive linear operator depending measurably on $\omega \in \Omega$.
The definition of the {\em generalized exponential separation\/} (of type I) is almost the same as the definition of the exponential separation for topological semiflows, the only difference being that $E(\omega)$ and $F(\omega)$ now depend measurably on $\omega$.  Also, the {\em generalized principal Lyapunov exponent\/} is the largest Lyapunov exponent for $\PP$-a.e. $\omega \in \Omega$.
\par\smallskip
Arnold \emph{et al.} in~\cite{AGD} were the first to prove the existence of generalized exponential separation for discrete-time positive random dynamical systems generated by random families of positive matrices.  Later, \JM\ and Shen~\cite{MiShPart1} provided  the assumptions required for general random positive linear skew-product semiflows (with both discrete and continuous time) in order to admit generalized principal Floquet subspaces and generalized exponential separation of type I (in~contrast to~\cite{MiSh1,MiSh3}, no embedding into topological semiflows was used in the proofs). The application of this theory to a variety of random dynamical systems arising from Leslie matrix models, cooperative linear ordinary differential equations and linear parabolic partial differential equations can be found in \JM\ and Shen~\cite{MiShPart2,MiShPart3}. In particular, the existence of generalized principal Floquet subspaces for  random skew-product semiflows generated by cooperative families of delay differential equations is also obtained in~\cite{MiShPart3}.
\par
\smallskip
In this paper, the positivity conditions satisfied by the linear random dynamical systems are weakened, in order to assure the existence of generalized principal Floquet subspaces and the existence of generalized exponential separations of type~II.
\par\smallskip
The structure and main results of the paper is as follows.
In Section~\ref{sec-prel} the notions and assumptions used throughout the paper are introduced. In particular, $X$ will be an ordered separable Banach space with dual $X^*$ separable and   positive cone $X^+$ normal and reproducing. Conditions of integrability and positivity for the random linear skew-product semiflow are imposed, and a new  focusing assumption is considered. More precisely, there is a positive time  $T>0$  such that if $u \in X^+$ and  $\omega \in \Omega$ then  $U_\omega (T)\,u =0$ or  $U_\omega (T)\,u $ is strictly positive and satisfies  a classical focusing inequality in the terms stated by \JM\ and Shen~\cite{MiShPart1}. For simplicity we fix the scale $T=1$ throughout the  paper. From these assumptions,  the integrability,  positivity and an alternative focusing property for the measurable dual skew-product semiflow are obtained.
 \par\smallskip
Under the new focusing condition, in Section~\ref{sec-GPFS} the existence of a family of generalized principal Floquet subspaces is shown. Section~\ref{sec-GES} is devoted to the introduction of the new concept of generalized exponential separation of type~II and the proof of the existence under the previously considered assumptions. Finally, Section~\ref{sec-delay} illustrates the application of the theory to random dynamical systems generated by scalar linear random delay differential equations with finite delay.
\section{Preliminaries}\label{sec-prel}
A {\em probability space\/} is a triple $\OFP$, where $\Omega$ is a set, $\mathfrak{F}$ is a $\sigma$\nobreakdash-\hspace{0pt}algebra of subsets of $\Omega$, and $\PP$ is a probability measure defined for all $F \in \mathfrak{F}$.  We always assume that the measure $\PP$ is complete.
\par\smallskip
A \emph{measurable dynamical system} on the probability space $\OFP$ is a $(\mathfrak{B}(\R) \otimes \mathfrak{F},\mathfrak{F})$\nobreakdash-\hspace{0pt}measurable mapping $\theta\colon\R\times \Omega\to \Omega$ such that
\begin{itemize}
\item
$\theta(0,\w)=\w$ for any $\w\in\Omega$,
\item
$\theta(t_1+t_2,w)=\theta(t_2,\theta(t_1,\w))$ for any $t_1$, $t_2\in\R$ and any $\w \in\Omega$.
\end{itemize}
We write $\theta(t,\w)$ as $\theta_t\w$. Also, we usually denote measurable dynamical systems by $(\OFP,(\theta_{t})_{t \in \R})$ or simply by $(\theta_{t})_{t \in \R}$.\par
A \emph{metric dynamical system} is a measurable dynamical system $(\OFP,(\theta_{t})_{t \in \R})$
such that for each $t\in\R$  the mapping $\theta_t\colon \Omega\to\Omega$ is $\PP$-preserving (i.e., $\PP(\theta_t^{-1}(F))=\PP(F)$ for any $F\in\mathfrak{F}$ and $t\in\R$).
\subsection{Measurable linear skew-product semidynamical systems}\label{subsec-measurable}
We consider a separable Banach space $X$ such that its dual $X^{*}$ is separable.
\par
\smallskip
We write $\R^{+}$ for $[0, \infty)$.  By a {\em measurable linear skew-product semidynamical system or semiflow},
$\Phi = \allowbreak ((U_\w(t))_{\w \in \Omega, t \in \R^{+}}, \allowbreak (\theta_t)_{t\in\R})$ on  $X$ covering a metric dynamical system $(\theta_{t})_{t \in \R}$ we understand a $(\mathfrak{B}(\R^{+}) \otimes \mathfrak{F} \otimes \mathfrak{B}(X), \mathfrak{B}(X))$\nobreakdash-\hspace{0pt}measurable
mapping
\begin{equation*}
[\, \R^{+} \times \Omega \times X \ni (t,\w,u) \mapsto U_{\w}(t)\,u \in X \,]
\end{equation*}
satisfying
\begin{align}
&U_{\w}(0) = \mathrm{Id}_{X} \quad & \textrm{for each }\,\w  \in \Omega, \label{eq-identity}\\
&U_{\theta_{s}\w}(t) \circ U_{\w}(s)= U_{\w}(t+s) \qquad &\textrm{for each } \,\w \in \Omega \textrm{ and }  t,\,s \in \R^{+},\label{eq-cocycle}\\
&[\, X \ni u \mapsto U_{\w}(t)u \in X \,] \in \mathcal{L}(X) & \textrm{for each }\,\w \in \Omega \textrm{ and } t \in \R^{+}.\nonumber
\end{align}
Sometimes we write simply $\Phi = ((U_\w(t)), \allowbreak (\theta_t))$. Eq.~\eqref{eq-cocycle} is called the {\em cocycle property\/}.
\par\smallskip
For $\w \in \Omega$, by an {\em entire orbit\/} of $U_{\omega}$ we understand a mapping $v_{\omega} \colon \R \to X$ such that $v_\w(s + t) = U_{\theta_{s}\w}(t)\, v_\w(s)$ for each $s \in \R$ and $t\geq 0$.
\par\smallskip
Next we introduce the  {\em dual\/} of $\Phi$. For $\w \in \Omega$, $t \in \R^{+}$ and $u^{*} \in X^*$ we define $U^{*}_{\w}(t)\,u^{*}$~by
\begin{equation}
\label{dual-definition}
\langle u, U^{*}_{\w}(t)\,u^{*} \rangle = \langle U_{\theta_{-t}\w}(t)\,u , u^{*} \rangle \qquad \text{for each } u \in X
\end{equation}
(in other words, $U^{*}_{\w}(t)$ is the mapping dual to $U_{\theta_{-t}\w}(t)$).
\par\smallskip
As explained in~\cite{MiShPart1}, since $X^{*}$ is separable, the mapping
\begin{equation*}
[\, \R^{+} \times \Omega \times X^{*} \ni (t,\w,u^{*}) \mapsto U^{*}_{\w}(t)\,u^{*} \in X^{*} \,]
\end{equation*}
is $(\mathfrak{B}(\R^{+}) \otimes \mathfrak{F} \otimes \mathfrak{B}(X^{*}),
\mathfrak{B}(X^{*}))$\nobreakdash-\hspace{0pt}measurable. The \mlsps\ $\Phi^{*} = ((U^{*}_\w(t))_{\w \in \Omega, t \in \R^{+}},(\theta_{-t})_{t\in\R})$ on $X^{*}$ covering $(\theta_{-t})_{t \in \R}$ will be called  the {\em dual\/} of $\Phi$. The cocycle property for the dual takes the form
\begin{equation}
\label{eq-cocycle-dual}
U^{*}_{\theta_{-t}\w}(s) \circ U^{*}_{\w}(t)= U^{*}_{\w}(t+s) \qquad \textrm{for each } \,\w \in \Omega \textrm{ and }  t,\,s \in \R^{+}
\end{equation}
\par
\smallskip
Let $\Omega_0 \in \mathfrak{F}$.
A family $\{E(\w)\}_{\w \in \Omega_0}$ of $l$\nobreakdash-\hspace{0pt}dimensional vector subspaces of $X$ is {\em measurable\/} if there are $(\mathfrak{F},
\mathfrak{B}(X))$-measurable functions $v_1, \dots, v_l \colon \Omega_0 \to X$ such that $(v_1(\w), \dots, v_l(\w))$ forms a basis of $E(\w)$ for each $\w \in \Omega_0$.
\par\smallskip
Let $\{E(\w)\}_{\w \in \Omega_0}$ be a family of $l$\nobreakdash-\hspace{0pt}dimensional vector subspaces of $X$, and let $\{F(\w)\}_{\w \in \Omega_0}$ be a family of $l$\nobreakdash-\hspace{0pt}codimensional closed vector subspaces of $X$, such that $E(\w) \oplus F(\w) = X$ for all $\w \in \Omega_0$.  We define the {\em family of projections associated with the decomposition\/} $E(\w) \oplus F(\w) = X$ as $\{P(\w)\}_{\w \in \Omega_0}$, where $P(\w)$ is the linear projection of $X$ onto $F(\w)$ along $E(\w)$, for each $\w \in \Omega_0$.
\par\smallskip
The family of projections associated with the decomposition $E(\w) \oplus F(\w) = X$ is called {\em strongly measurable\/} if for each $u \in X$ the mapping $[\, \Omega_0 \ni \w \mapsto P(\w)u \in X \,]$ is $(\mathfrak{F}, \mathfrak{B}(X))$\nobreakdash-\hspace{0pt}measurable.
\par\smallskip
We say that the decomposition $E(\w) \oplus F(\w) = X$, with
$\{E(\w)\}_{\w \in \Omega_0}$ finite\nobreakdash-\hspace{0pt}dimensional, is {\em invariant\/} if $\Omega_0$ is invariant, $U_{\w}(t)E(\w) = E(\theta_{t}\w)$
and $U_{\w}(t)F(\w) \subset F(\theta_{t}\w)$, for each $t \in \T^{+}$.
\par\smallskip
A strongly measurable family of projections associated with the invariant decomposition $E(\w) \oplus F(\w) = X$ is referred to as {\em tempered\/} if
\begin{equation}\label{tempered}
\lim\limits_{t \to \pm\infty} \frac{\ln{\n{P(\theta_{t}\w)}}}{t} = 0 \qquad \PP\text{-a.e. on }\Omega_0.
\end{equation}
\subsection{Ordered Banach spaces}\label{subsec-OBS} Let $X$ be a Banach space with norm $\n{\cdot}$. We say that $X$ is an {\em ordered Banach space\/} if there is a closed convex cone, that is, a nonempty closed subset $X^+\subset X$ satisfying
\begin{itemize}
\item[(O1)] $X^++X^+\subset X^+$.
\item[(O2)] $\R^+ X^+\subset X^+$.
\item[(O3)] $X^+\cap (-X^+)=\{0\}$.
\end{itemize}
Then  a partial ordering in $X$ is defined by
\begin{align*}
u  \leq v \quad& \Longleftrightarrow \quad v -u\in X_+\,;\\
 u< v    \quad  & \Longleftrightarrow \quad v-u\in X_+
 \text{ and }\; u\neq v\,.
\end{align*}
The cone $X^+$ is said to be {\em reproducing\/} if $X^+-X^+=X$. The cone $X^+$ is said to be {\em normal} if the norm of the Banach space $X$ is {\em semimonotone\/}, i.e., there is a positive constant $k>0$ such that $0\leq u\leq v$ implies $\n{u}\leq k\,\n{v}$. In such a case, the Banach space can be renormed so that for any $u$, $v\in X$, $0\leq u\leq v$ implies $\n{u}\leq\n{v}$ (see Schaefer~\cite[V.3.1, p. 216]{Schaef}). Such a  norm  is called {\em monotone}.
\par
\smallskip
For an ordered Banach space $X$ denote by $(X^{*})^{+}$ the set of all $u^{*} \in X^{*}$ such that $\langle u, u^{*} \rangle \ge 0$ for all $u \in X^{+}$.  The set $(X^{*})^{+}$ has the properties of a cone, except that $(X^{*})^{+} \cap (-(X^{*})^{+}) = \{0\}$ need not be satisfied (such sets are called {\em wedges\/}).
\par
\smallskip
If $(X^{*})^{+}$ is a cone we call it the {\em dual cone}.  This happens, for instance, when $X^{+}$ is total (that is, $X^{+}- X^{+}$ is dense in $X$, which in particular holds when $X^+$ is reproducing and this will be one of our hypothesis).
\par\smallskip
Nonzero elements of $X^{+}$ (resp.~of $(X^{*})^{+}$) are called {\em positive\/}.
We say that two positive vectors $u$, $v\in X^+\setminus\{0\}$ are {\em comparable\/}, written $u\sim v$, if there are positive numbers, $\underline{\alpha}$, $\overline{\alpha}$, such that $\underline{\alpha}\, v \le u \le \overline{\alpha} \,v$. For a nonzero vector $u\in X^+$ we call the {\em component of $u$\/}
\begin{equation}\label{component}
C_u=\{v\in X^+\setminus \{0\}\mid v\sim u\}\,,
\end{equation}
i.e., the equivalence class of $u$.\par\smallskip
We now recall the concept of the Hilbert projective metric.
\begin{defi}\label{ProjectMetric} Let $u$, $v\in X$.
\begin{itemize}
\item[(1)] If $\{\underline{\alpha}\in\R \mid \underline{\alpha}\, v \le u\}$ is nonempty, define
    \[m(u/v) := \sup \{\underline{\alpha}\in\R \mid \underline{\alpha}\, v \le u\}\,.\]
 If $\overline{\alpha}\in\R\mid \{u \le \overline{\alpha} \,v\}$ is nonempty, define
    \[ M(u/v):=\inf \{\overline{\alpha}\in\R\mid  u\le \overline{\alpha} \,v\}\,.\]
 \item[(2)] If both $m(u/v)$ and $M(u/v)$ exist, define the {\em oscillation\/} of $u$ over $v$, and  if $u\sim v$ the {\em projective distance\/} between $u$ and $v$ as
 \begin{equation}\label{projdist}
 \textit{osc}\,(u/v):=M(u/v)-m(u/v)\quad\text{and}\quad d(u,v):=\ln\frac{M(u/v)}{m(u/v)}\,.
 \end{equation}
\end{itemize}
\end{defi}
\par
For the next result, see~\cite[Lemma~4.6]{MiShPart1}.
\begin{lema}
\label{lm:projective-vs-norm}
Assume that $X^{+}$ is normal.  Then for any $u, v \in X^{+}$, $u \sim v$, with $\lVert u \rVert = \lVert v \rVert = 1$, there holds
\begin{equation*}
\lVert u-v \rVert \le 3 \bigl( e^{d(u,v)}-1 \bigr).
\end{equation*}
\end{lema}
\smallskip
\subsection{Assumptions}\label{sec-assumptions} Throughout the paper we will assume that $X$ is an ordered separable  Banach space with $\dim X\geq 2$ such that its dual $X^{*}$ is separable, with positive cone $X^+$ normal and reproducing.  It follows then that the dual cone $(X^{*})^{+}$ is normal and reproducing, too (see~\cite[V.3]{Schaef}). We always assume that the norms on $X$ and on $X^{*}$ are monotone.
\par
\smallskip
Let $\Phi = ((U_\omega(t)), (\theta_t))$ be a \mlsps\ on $X$ covering an ergodic metric dynamical system $(\theta_{t})$ on $\OFP$, and its dual $\Phi^{*} = ((U^{*}_\w(t))_{\w \in \Omega, t \in \R^{+}},(\theta_{-t})_{t\in\R})$ on $X^{*}$ covering $(\theta_{-t})_{t \in \R}$.  \par\smallskip
We now list assumptions we will make at various points in the sequel.
\par\smallskip
\noindent (\textbf{A1}) (Integrability) The functions
\begin{equation*}
\bigl[ \, \Omega \ni \omega \mapsto \sup\limits_{0 \le s \le 1} {\lnplus{\n{U_{\omega}(s)}}} \in [0,\infty) \, \bigr] \in L_1\OFP\;\; \text{ and}
\end{equation*}
\begin{equation*}
\bigl[ \, \Omega \ni \omega \mapsto \sup\limits_{0 \le s \le 1}
{\lnplus{\n{U_{\theta_{s}\omega}(1-s)}}} \in [0,\infty) \, \bigr] \in L_1\OFP.
\end{equation*}
\par\smallskip
\noindent (\textbf{A1})* (Integrability) The functions
\begin{equation*}
\bigl[ \, \Omega \ni \omega \mapsto \sup\limits_{0 \le s \le 1} {\lnplus{\n{U_{\w}^*(s)}}} \in [0,\infty) \, \bigr] \in L_1\OFP\;\; \text{ and}
\end{equation*}
\begin{equation*}
\bigl[ \, \Omega \ni \omega \mapsto \sup\limits_{0 \le s \le 1}
{\lnplus{\n{U_{\theta_{s}\w}^*(1-s)}}} \in [0,\infty) \, \bigr] \in L_1\OFP.
\end{equation*}
\par\smallskip
\noindent (\textbf{A2}) (Positivity) For any $\omega \in \Omega$, $t \ge 0$ and $u_1, u_2 \in X$ with $u_1 \le u_2$
\begin{equation*}
U_{\omega}(t)\,u_1 \le U_{\omega}(t)\,u_2\,.
\end{equation*}
\par\smallskip
\noindent (\textbf{A2})* (Positivity)  For any $\omega \in \Omega$, $t \ge 0$ and $u^{*}_1, u^{*}_2 \in X^*$ with $u^{*}_1 \le u^{*}_2$
\begin{equation*}
U^{*}_{\omega}(t)\,u^{*}_1 \le U^{*}_{\omega}(t)\,u^{*}_2\,.
\end{equation*}
Notice that in our case, as explained in~\cite{MiShPart1}, (A1)* follows from (A1) and (A2)* follows from (A2).
\par
\smallskip
For a \mlsps\ $\Phi$ satisfying (A2) we say that an entire orbit $v_{\w}$ of $U_{\w}$ is {\em positive\/} if $v(\w) \in X^{+} \setminus \{0\}$ for all $t \in \R$.  Similarly, when (A2)* is satisfied, an entire orbit $v^{*}_{\w}$ of $U^{*}_{\w}$ is {\em positive\/} if $v^{*}(\w) \in (X^{*})^{+} \setminus \{0\}$ for all $t \in \R$.
\par
\medskip
Next we introduce focusing conditions (A3) and (A3)* in the following way.
\par\smallskip
\noindent (\textbf{A3}) (Focusing) (A2) is satisfied  and there are $\be\in X^+$ with $\|\be\|=1$ and $U_\w(1)\,\be\neq 0 $ for all $\w\in\Omega$, and an $(\mathcal{F},\mathcal{B}(\R))$\nbd-measurable function $\varkappa\colon\Omega\to [1,\infty)$ with
$\lnplus\ln\varkappa\in L_1((\Omega,\mathcal{F},\PP))$ such that for any  $\w\in\Omega$ and any nonzero $u\in X^+$
\begin{itemize}
\item $U_\w(1)\,u=0$, or
\item $U_\w(1)\,u\neq 0$ and there is $\beta(\w,u)>0$ with the property that
\begin{equation}\label{focusing}
\beta(\w,u)\,\be\leq U_\w(1)\,u\leq \varkappa(\w)\,\beta(\w,u)\,\be\,.
\end{equation}
\end{itemize}
\begin{nota}\label{rm:dichotomy}
Under~(A3), by the cocycle property~\eqref{cocycle2}, for $u \in X^{+}$ the following dichotomy holds:
\begin{itemize}
\item
$U_\w(t)\, u = 0$ for all $t \ge 1$, or
\item
$U_\w(t)\, u > 0$ for all $t \ge 0$.
\end{itemize}
\end{nota}
\medskip
\noindent (\textbf{A3})* (Focusing for $X^*$) (A2)* is satisfied and there are ${\bf e^*}\in (X^*)^+$ with $\langle \be,\be^*\rangle=1$ and $\| \be^*\|=1$ and an $(\mathcal{F},\mathcal{B}(\R))$-measurable function $\varkappa^*\colon \Omega\to[1,\infty)$ with
$\lnplus\ln\varkappa^*\in L_1((\Omega,\mathcal{F},\PP))$ such that for any  $\w\in\Omega$ there holds $U_\w^*(1)\, \be^* \ne 0$, and for any $\w\in\Omega$ and any nonzero $u^*\in (X^*)^+$ there is $\beta^*(\w,u^*)>0$ with the property that
\begin{equation}\label{focusing*}
\beta^*(\w,u)\,U_\w^*(1)\,{\bf e^*}\leq U_\w^*(1)\,u^*\leq \varkappa^*(\w)\,\beta^*(\w,u^*)\,U^*_\w(1)\,\be^*\,.
\end{equation}
\par
\smallskip
It should be remarked that our condition (A3) is weaker than condition (A3) in~\cite{MiShPart1}.  We write the latter as:
\par
\smallskip
\noindent (\textbf{A3-O}) (A2) is satisfied  and there are $\be \in X^+$ with $\|\be\| = 1$ and $U_\w(1) \, \be \neq 0 $ for all $\w \in \Omega$, and an $(\mathcal{F},\mathcal{B}(\R))$\nbd-measurable function $\varkappa \colon \Omega \to [1,\infty)$ with
$\lnplus \ln\varkappa \in L_1((\Omega,\mathcal{F},\PP))$ such that for any  $\w \in \Omega$ and any nonzero $u \in X^+$ there holds $U_\w(1) \, u \neq 0$ and there is $\beta(\w,u) > 0$ with the property that
\begin{equation}
\label{focusing-old}
\beta(\w,u) \, \be \leq U_\w(1) \,u \leq \varkappa(\w) \,\beta(\w,u)\,\be\,.
\end{equation}
\par
\smallskip
Condition (A3)* follows from (A3), as we prove now.
\begin{prop}\label{A3->A3*} {\rm(\textbf{A3}) $\Rightarrow$ (\textbf{A3})* }.
\end{prop}
\begin{proof}
From~\eqref{focusing} we have
\begin{equation*}
\beta(\theta_{-1}\w,u)\,\be\leq U_{\theta_{-1}\w}(1)\,u\leq \varkappa(\theta_{-1}\w)\,\beta(\theta_{-1}\w,u)\,\be
\end{equation*}
provided that $U_{\theta_{-1}\w}(1) \, u \neq 0$. Fix any ${\bf e^*} \in (X^*)^+$ with $\langle \be,\be^* \rangle = 1$ and $\| \be^*\| = 1$ (the existence of such an $\mathbf{e}^{*}$ follows from \cite[Theorem~5.4]{Schaef}).  Hence,
\begin{equation*}
\beta(\theta_{-1}\w,u)\,\langle \be,\be^*\rangle\leq \langle U_{\theta_{-1}\w}(1)\,u, \be^*\rangle=\langle u, U^*_\w(1)\,\be^*\rangle
\end{equation*}
(in~particular, $U_\w^*(1)\, \be^* \ne 0$) and
\begin{equation*}
\langle U_{\theta_{-1}\w}(1)\,u,u^*\rangle \leq \varkappa(\theta_{-1}\w)\,\beta(\theta_{-1}\w,u)\,\langle \be, u^*\rangle\,\langle \be, \be^*\rangle\,.
\end{equation*}
Combining these two inequalities and denoting
\[\beta^*(\w,u^*)=\frac{\langle\be, u^*\rangle}{\varkappa(\theta_{-1}\w)}\;\text{ and }\; \varkappa^*(\w)=\varkappa^2(\theta_{-1}\w)\,,
\]
we obtain
\begin{equation*}
\langle U_{\theta_{-1}\w}(1)\,u,u^*\rangle \leq \varkappa^*(\w)\,\beta^*(\w,u^*)\langle u, U^*_\w(1)\,\be^*\rangle\,,
\end{equation*}
which also holds if $ U_{\theta_{-1}\w}(1)\,u=0$, and then
\begin{equation}\label{derecha}
U_\w^*(1)\,u^*\leq \varkappa^*(\w)\,\beta^*(\w,u^*)\,U^*_\w(1)\,{\bf{e^*}}\,.
\end{equation}
Analogously, from~\eqref{focusing} and $\langle\be,\be^*\rangle=1$ we deduce that
\begin{align*}
\langle U_{\theta_{-1}\w}(1)\,u, u^*\rangle & \geq \beta(\theta_{-1}\w,u)\,\langle\be,u^*\rangle = \frac{\beta(\theta_{-1}\w,u)}{\varkappa(\theta_{-1}\w)}\,\langle \be,u^*\rangle\,\varkappa(\theta_{-1}\w)\\
& \geq  \frac{\langle\be,u^*\rangle}{\varkappa(\theta_{-1}\w)}\,\langle U_{\theta_{-1}\w}(1)\,u,\be^*\rangle=\beta^*(\w,u^*)\,\langle U_{\theta_{-1}\w}(1)\,u,\be^*\rangle\,,
\end{align*}
that is,
\begin{equation*}
\beta^*(\w,u)\,U_\w^*(1)\,{\bf e^*}\leq U_\w^*(1)\,u^*\,,
\end{equation*}
which together with~\eqref{derecha} finishes the proof.
\end{proof}
A simple consequence is the following. As usual, we denote by $\lfloor t \rfloor$ the integer part of the real number $t$.
\begin{lema}
\label{lm:dual-injective}
Assume~\textup{(A3)}. Then for any $\omega \in \Omega$, any $t \ge 0$ and any nonzero $u^{*} \in (X^{*})^{+}$ there holds $U^{*}_{\omega}(t) \, u^{*} \ne 0$.
\end{lema}
\begin{proof}
It follows from Proposition~\ref{A3->A3*}, by induction, that $U^{*}_{\omega}(n) \, u^{*} \ne 0$ for  each $n\in\N$.  Suppose to the contrary that $U^{*}_{\omega}(t) \, u^{*} = 0$ for some $\omega \in \Omega$, $t > 0$ and $u^{*} \in (X^{*})^{+}$.  Then, by~\eqref{eq-cocycle-dual}, $U^{*}_{\omega}(\lfloor t \rfloor + 1) \, u^{*} = U^{*}_{\theta_{-t}\omega}(\lfloor t \rfloor + 1 - t) \, U^{*}_{\omega}(t) \, u^{*} = 0$, a contradiction.
\end{proof}
\begin{nota}\label{timeT}
We can replace time 1 with some nonzero $T\in\R^+$ in (A1), (A3), and  (A1)*, (A3)*.
\end{nota}
\section{Generalized principal Floquet subspaces}\label{sec-GPFS}
The main result of the paper is the proof of the existence, under the new focusing assumption (A3), of a new version of the concept of generalized exponential separation for  a \mlsps.
We will preserve the structure and follow the arguments in~\cite{MiShPart1}, just introducing
the modifications which are required in this new situation. This section  will be devoted to the proof of the existence of a family of generalized principal Floquet subspaces.

\smallskip
As stated before, $X$ is an ordered separable Banach space such that $X^{*}$ is separable, with positive cone $X^+$ normal and reproducing, and recall that, from Lemma~\ref{A3->A3*}, once (A3) is assumed, assumption (A3)* holds.
\begin{defi}
\label{generalized-floquet-space}
A family of one\nobreakdash-\hspace{0pt}dimensional subspaces $\{\widetilde{E}(\w)\}_{\w \in \widetilde{\Omega}}$ of $X$ is called a family of {\em generalized principal Floquet subspaces} of $\Phi = ((U_\w(t)), (\theta_t))$ if $\widetilde{\Omega} \subset \Omega$ is invariant, $\PP(\widetilde{\Omega}) = 1$, and
\begin{itemize}
\item[{\rm (i)}]
$\widetilde{E}(\w) = \spanned{\{w(\w)\}}$ with $w \colon \widetilde{\Omega} \to X^+ \setminus \{0\}$ being $(\mathfrak{F}, \mathfrak{B}(X))$\nobreakdash-\hspace{0pt}measurable,
\item[{\rm (ii)}]
$U_{\w}(t) \widetilde{E}(\w) = \widetilde{E}(\theta_{t}\w)$, for any $\w \in \widetilde{\Omega}$ and any $t > 0$,
\item[{\rm (iii)}]
there is $\widetilde{\lambda} \in [-\infty, \infty)$ such that
\begin{equation*}
\widetilde{\lambda} = \lim_{t\to\infty} \frac{\ln{\n{U_\w(t)\,w(\w)}}}{t} \quad \text{ for each } \w \in \widetilde{\Omega},
\end{equation*}
\item[{\rm (iv)}] for each $\w\in\widetilde\Omega$ with $U_\w(1)\,u\neq 0$
\begin{equation*}
\limsup_{t\to\infty} \frac{\ln{\n{U_\w(t)\,u}}}{t} \le \widetilde{\lambda} \,.
\end{equation*}
\end{itemize}
$\widetilde{\lambda}$ is called the {\em generalized principal Lyapunov exponent} of $\Phi$ associated to the generalized principal Floquet subspaces $\{\widetilde{E}(\w)\}_{\w\in\widetilde{\Omega}}$.
\end{defi}
\par\smallskip
We recall the definitions of oscillation ratio, Birkhoff contraction ratio and projective diameter, needed in the proofs of our main theorems. See Definition~\ref{ProjectMetric} and Subsection~\ref{subsec-OBS} for previous definitions and notations.
\begin{defi}\label{OscilRatio}
Assume (A2) and let $\w\in\Omega$.
\begin{itemize}
\item[(1)] The {\em oscillation ratio\/} of $U_\w(1)$ is defined as
\[p(\w):= \sup_{\substack{u,\,v\in X^+\\ u\sim v\,,\; u\neq \alpha\,v}}\frac{\textit{osc}\,(U_\w(1))\,u/U_\w(1)\,v)}{\textit{osc}\,(u/v)}\,.\]
\item[(2)]  The {\em Birkhoff contraction ratio\/} of $U_\w(1)$ is defined as
\[q(\w):=\sup_{\substack{u,\,v\in X^+\\ u\sim v\,,\; u\neq \alpha\,v}} \frac{d(U_\w(1)\,u,U_\w(1)\,v)}{d(u,v)}\,.\]
\item[(3)]
The {\em projective diameter} of $U_\w(1)$ is defined as
\[\tau(\w):=\sup_{\substack{u,\,v\in X^+ \\ U_\w(1)\,u \sim U_\w(1)\,v}} d(U_\w(1)\,u,U_\w(1)\,v)\,.\]
\end{itemize}
The functions $p^*$, $q^*$ and $\tau^*$ for the dual $\Phi^*$ are defined in an analogous way.
\end{defi}
The following lemmas, proved in~\cite{MiShPart1} as Lemma 4.10 and 4.11 for different focusing conditions, hold for the new focusing conditions (A3) and (A3)*  with small changes, and hence their proofs are omitted. Recall that $\be\in X^+$ and $\be^*\in (X^*)^+$ are the positive vectors of condition (A3) and (A3)*, respectively, and notice that it is easy to check that
\begin{equation}\label{tau}
\tau(\w)=\sup_{\substack{u,\,v\in X^+ \\ U_\w(1)\,u \sim U_\w(1)\,v\\ U_\w(1)\,u\sim \be}} d(U_\w(1)\,u,U_\w(1)\,v)\,.
\end{equation}
\begin{lema}
\label{lemma4.10}
Under assumption {\rm (A3)}, for each $\w\in\Omega$ and each $u\in X^+$, such that $U_\w(1)\,u\sim \be$ and $u^*\in (X^*)^+\setminus \{0\}$ there holds  $U^*_\w(1)\,u^*\sim U^*_\w(1)\,\be^*$ and
\[d(U_\w(1)\,u,\be)\leq \ln \varkappa(\w), \quad d(U^*_\w(1)\,u^*,U^*_\w(1)\,\be^*)\leq \ln\varkappa^*(\w)\,. \]
Consequently, $\tau(\w)\leq 2\,\ln\varkappa(\w)$ and $\tau^*(\w)\leq 2\,\ln\varkappa^*(\w)$ for any $\w\in\Omega$.
\end{lema}
\begin{lema} Under assumption {\rm (A3)}, for each $\w\in\Omega$
\[\tau(\w)<\infty\,,\quad \tau^*(\w)<\infty\,,\quad \text{and}\]
\[ p(\w)=q(\w)=\tanh\frac{\tau(\w)}{4}\;\; (<1)\,,\quad p^*(\w)=q^*(\w)=\tanh\frac{\tau^*(\w)}{4}\;\; (<1)\,.\]
\end{lema}
As in Lemma 4.13 of~\cite{MiShPart1}, taking into account~\eqref{tau} and changing the dense countable subsets of the proof by $(v_j+\be/n)_{j,n}$,  the following result is proved.
\begin{lema}
Under assumption {\rm (A3)}, the functions
\begin{align*}
\bigl[ \, \Omega \ni \omega \mapsto \tau(\w)\in\R \,\bigr],&\quad \bigl[ \, \Omega \ni \omega \mapsto \tau^*(\w)\in\R \,\bigr]\\
\bigl[ \, \Omega \ni \omega \mapsto p(\w)\in\R\,\bigr],& \quad \bigl[ \, \Omega \ni \omega \mapsto p^*(\w)\in\R \,\bigr]\\
\bigl[ \, \Omega \ni \omega \mapsto q(\w)\in\R\,\bigr],& \quad \bigl[ \, \Omega \ni \omega \mapsto q^*(\w)\in\R \,\bigr]
\end{align*}
are $(\mathfrak{F},\mathfrak{B}(\R))$\nbd-measurable.
\end{lema}
Next we consider the set of positive vectors
\begin{equation*}
\Sigma := \{u \in C_\be \mid \n{u} = 1 \} = C_\be \cap S_1(X^+)\,,
\end{equation*}
where $C_\be$ denotes the component of $\be$ defined in~\eqref{component}.
\begin{lema}
\label{lm:Sigma-complete}
Assume~\textup{(A3)}.  Then $d$ is a metric on $\Sigma$, and $(\Sigma, d)$ is a complete metric space.
\end{lema}
\begin{proof}
As stated in Eveson~\cite{Evenson}, since $X^{+}$ is normal (hence almost Archimedean) and the norm on $X$ is monotone, any two vectors $u$, $v$ of $\Sigma\subset C_\be$ are {\em regularly comparable\/}, i.e., they are comparable and $m(u/v)\leq 1\leq M(u/v)$ (see Definition~\ref{ProjectMetric} for notation).  Consequently, from~\cite[Theorem~1.2.1]{Evenson} the projective distance $d$ defined in~\eqref{projdist} is a metric on $\Sigma$ and the metric space $(\Sigma, d)$ is complete.
\end{proof}
\begin{lema}
\label{lm:comparable-with-e}
Under assumption {\rm (A3)},
\[ U_\w(t)\,u\in C_\be\quad \text{ for  each}\;\; t\geq 0\,,\;\w\in\Omega \;\;\text{and} \;\;u\in C_\be\,.\]
\end{lema}
\begin{proof}
First we check that $U_\w(t)\,\be\in C_\be$ for each $t\geq 0$. From (A3) we know that $U_\w(1)\,\be\neq 0$, and then~\eqref{focusing} implies that $U_\w(1)\,\be\in C_\be$ for each $\w\in\Omega$. From the cocycle property~\eqref{eq-cocycle} we deduce that $U_\w(2)\,\be=U_{\theta_1\w}(1)\,U_\w(1)\,\be$, which together with (A2) yields $U_\w(2)\,\be\in C_\be$. In a recursive way we obtain that
\begin{equation}
\label{natural}
U_\w(n)\,\be\in C_\be\quad \text{ for each } \w\in\Omega \text{ and } n\in\N\,.
\end{equation}
Now take $t\geq 1$. Since $U_\w(t)\,\be=U_{\theta_{t-1}\w} (1)\,U_\w(t-1)\,\be\,,$
from the focusing condition (A3) we have two options $U_\w(t)\,\be=0$ or  $U_\w(t)\,\be\in C_\be$.
 Assume that $U_\w(t)\,\be=0$ and take $n_1\in\N$ such that $n_1\leq t\leq n_1+1$. Again~\eqref{eq-cocycle} provides
\[U_\w(n_1+1)\,\be=U_{\theta_{t}\w}(n_1+1-t)\,U_\w(t)\,\be=0\,,\]
which contradicts~\eqref{natural} and proves that $U_\w(t)\,\be\in C_\be$ for $t\geq 1$.\par\smallskip
Next, if $t\geq 0$, we have proved that $U_{\theta_{-1}\w}(1)\,\be$ and $U_{\theta_{-1}\w}(t+1)\,\be\in C_\be$, that is,
\[\widetilde \alpha\, \be\leq U_{\theta_{-1}\w}(1)\,\be\leq \widetilde\beta\, \be \quad \text{and} \quad  \alpha\, \be\leq  U_{\theta_{-1}\w}(t+1)\,\be\leq \beta\, \be\,, \]
and consequently, from $U_{\theta_{-1}\w}(t+1)\,\be=U_\w(t)\,U_{\theta_{-1}\w}(1)\,\be$  and the monotonicity property (A2) we deduce that
\[\alpha\, \be\leq \widetilde\beta \,U_\w(t)\,\be\quad \text{and}\quad \beta\,\be\geq \widetilde\alpha \, U_\w(t)\,\be\,\]
i.e., $U_\w(t)\,\be\in C_\be$, as claimed.\par\smallskip
Finally, if $u\in C_\be$, it is immediate to check from $U_\w(t)\,\be\in C_\be$ and (A2) that  $U_\w(t)\,u\in C_\be$, which finishes the proof.
\end{proof}
As a consequence of this lemma, under assumption (A3), the map
\begin{equation}\label{defiw}
\begin{array}{lccc}
\mathcal U_\w(t)\colon & \Sigma &\longrightarrow & \Sigma \\[.1cm]
                        & u & \mapsto & \displaystyle{\frac{U_\w(t)\,u}{\n{U_\w(t)\,u}}}
\end{array}
\end{equation}
is well defined for each $t\geq 0$ and $\w\in\Omega$.  Moreover, it is continuous for the projective distance because
$d(\mathcal U_\w(t)\,u,\mathcal U_\w(t)\,v)\leq d(u,v)$, as can be easily deduced from Lemmas~4.5 and~4.9 of~\cite{MiShPart1}.  Furthermore,  as a consequence of~\eqref{eq-cocycle}
\begin{equation}
\label{cocycle2}
\mathcal U_{\theta_{s}\w}(t) \circ \mathcal U_{\w}(s)= \mathcal U_{\w}(t+s)\,,\quad \textrm{for each }\,\w \in \Omega \textrm{ and } s,\,t \in \R^{+}\,.
\end{equation}
\par\smallskip
We omit the proof of the next result which is completely analogous to the first part of the proof of Proposition~5.3 of~\cite{MiShPart1}. It follows from~\eqref{cocycle2}, the properties of the distance, the definition of $q$ (see Definition~\ref{OscilRatio}) and Lemma~\ref{lemma4.10}. As before, we denote by $\lfloor t \rfloor$ the integer part of the real number $t$.
\begin{prop}\label{prop5.3} Under assumption {\rm (A3)},
\begin{itemize}
\item[(i)]
for each $\w\in\Omega$, $t\geq 2$ and $u$, $\widetilde u \in \Sigma$
\begin{equation*}
d\bigl(\mathcal U_{\theta_{-t}\w}(t) \, u, \mathcal U_{\theta_{-t}\w}(t)\,\widetilde u \big) \leq 2 \, \ln\varkappa\big(\theta_{-\lfloor t \rfloor} \w \bigr) \, q\bigl(\theta_{-\lfloor t \rfloor+1}\w \big) \cdots q\big(\theta_{-1}\w \bigr) \,,
\end{equation*}
\item[(ii)]
for each $\w \in \Omega$, $t \geq 2$ and $u$, $\widetilde u \in \Sigma$
\begin{equation*}
d\bigl(\mathcal U_\w(t) \, u, \mathcal U_\w(t) \, \widetilde u \bigr) \leq 2 \, \ln\varkappa(\w) \, q\bigl(\theta_1\w \bigr) \cdots q\bigl(\theta_{\lfloor t \rfloor-1}\w \bigr)\,.
\end{equation*}
\end{itemize}
\end{prop}
\smallskip
As a consequence, the following contraction property follows, whose proof is also omitted because it follows the arguments of Proposition~5.4 of~\cite{MiShPart1}.
\begin{prop}
\label{prop5.4}
Under assumption {\rm (A3)}, let $I:=\int_\Omega\ln q\, d\PP<0$. Then, there is an invariant set $\bar\Omega_1\subset \Omega$ with $\PP(\bar\Omega_{1}) = 1$ such that
\begin{itemize}
\item[(1)]
for each $I<J<0$ and  $\w\in\bar\Omega_1$, there is a $C_1(J,\w)>0$ such that
\[d\bigl(\mathcal U_{\theta_{-t}\w}(t)\,u,\, \mathcal U_{\theta_{-t}\w}(t)\,\widetilde u\bigr)\leq C_1(J,\w)\,e^{Jt}\]
whenever $t\geq 3$ and $u$, $\widetilde u\in\Sigma$,
\item[(2)]
for each $I<J<0$ and  $\w\in\bar\Omega_1$, there is a $C_2(J,\w)>0$ such that
\[d\bigl(\mathcal U_\w(t)\,u,\, \mathcal U_\w(t)\,\widetilde u\bigr)\leq C_2(J,\w)\,e^{Jt}\]
whenever $t\geq 2$ and $u$, $\widetilde u\in\Sigma$.
\end{itemize}
\end{prop}
\smallskip
Proposition~\ref{prop5.4}(1) ensures that for any $\w \in \bar\Omega_1$ the following exists
\begin{equation}
\label{w(w)}
w(\w) := \lim_{s\to\infty} \mathcal{U}_{\theta_{-s}\w}(s) \, \be,
\end{equation}
where the limit is taken in $d$.  Since, by Lemma~\ref{lm:Sigma-complete}, $(\Sigma, d)$ is a complete metric space, $w(\w)$ belongs to $\Sigma$.  Further, it follows from Lemma~\ref{lm:projective-vs-norm} that the above limit can be taken in the $X$\nobreakdash-\hspace{0pt}norm.  Moreover, since the functions $[\w\mapsto \mathcal U_{\theta_{-n}\w}(n)\,\be]$ are $(\mathfrak{F},
\mathfrak{B}(X))$-measurable, the function $w\colon \bar\Omega_1\to X$ is measurable.
\par\smallskip
The next theorem shows the existence of generalized Floquet subspaces and principal Lyapunov exponent, and the uniqueness of entire positive orbits, which is the equivalent result to Theorem~3.6 of~\cite{MiShPart1} for the new focusing.
\begin{teor}\label{teor3.6}
Under assumptions  {\rm (A1)} and {\rm (A3)}, if $\bar\Omega_1$ is the invariant set of {\rm Proposition~\ref{prop5.4}}, there is  an $(\mathfrak{F},\mathfrak{B}(X))$-measurable function
\[
w\colon\bar\Omega_1\to \Sigma=C_\be\cap S_1(X^+),\;\w\mapsto w(\w)
 \]
 defined by~\eqref{w(w)} such that
\begin{itemize}
\item[(1)] for each $\w\in\bar\Omega_1$ and $t\geq 0$,
\begin{equation}\label{prop_w}
w(\theta_t\w)=\frac{U_\w(t)\,w(\w)}{\n{U_\w(t)\,w(\w)}}\,;
\end{equation}
\item[(2)]
for each $\w\in\bar\Omega_1$, the map $w_\w\colon \R\to X^+$ defined by
\begin{equation}
\label{eq:definition-of-w}
w_\w(t)
= \begin{cases}
\displaystyle\frac{w(\theta_t\w)}{\n{U_{\theta_t\w}(-t) \, w(\theta_t\w)}} & \quad \text{for } t\leq 0\,,
\\[.4cm]
\;U_\w(t)\,w(\w)& \quad \text{for } t\geq 0\,; \end{cases}
\end{equation}
is a positive entire orbit of $U_\w$, unique up to multiplication by a positive scalar;
\item[(3)]
there are an invariant set $\widetilde\Omega_1 \subset \bar\Omega_1$ with $\PP(\widetilde\Omega_1) = 1$ and  a $\bar\lambda_1 \in [-\infty,\infty)$ such that
\begin{equation*}
\bar\lambda_1 = \lim_{t\to\pm\infty} \frac{1}{t}\ln\n{w_\w(t)} = \int_\Omega \ln\n{w_{\omega'}(1)}\,d\PP(\omega')
\end{equation*}
for each $\w \in \widetilde\Omega_1$;
\item[(4)]
for each $\w \in \widetilde\Omega_1$ and $u \in X^+$ with $U_\w(1) \,u \neq 0$
\begin{equation}
\label{eqLyap}
\lim_{t\to\infty} \frac{1}{t}\ln\n{U_\w(t)\,u} = \bar \lambda_1\,;
\end{equation}
\item[(5)]
for each $\w \in \widetilde\Omega_1$ and $u \in X$
\begin{equation}
\label{ineqLyap}
\limsup_{t\to\infty}\frac{1}{t}\ln\n{U_\w(t)\,u}\leq \bar \lambda_1\,;
\end{equation}
that is,  $\{E_1(\w)\}_{\w \in \widetilde\Omega_1}$, with $E_1(\w) = \spanned\{w(\w)\}$, is a family of generalized principal Floquet subspaces, and $\bar\lambda_1$ is the generalized principal Lyapunov exponent.
\end{itemize}
\end{teor}
\begin{proof} (1)  From relation~\eqref{cocycle2} and the definition of $w(\w)$  we deduce that
\begin{align*}
\mathcal U_\w(t)\,w(\w) &=\mathcal U_\w(t)\left( \lim_{s\to\infty} \mathcal U_{\theta_{-s}\w}(s)\,\be\right)=\lim_{s\to\infty} \bigl( \mathcal U_\w(t) \circ \mathcal U_{\theta_{-s}\w}(s) \bigr)\, \be
\\
& = \lim_{s\to\infty}\mathcal U_{\theta_{-s}\w}(s+t)\,\be=\lim_{s\to\infty} \bigl( \mathcal U_{\theta_{t-s}\w}(s) \circ \mathcal U_{\theta_{-s}\w}(t)\bigr) \, \be \,,
\end{align*}
for each $\w \in \bar\Omega_1$ and $t \geq 0$. Moreover, from Proposition~\ref{prop5.4}(1)
\[
d\bigl(\mathcal U_{\theta_{t-s}\w}(s) \, \be, \, \mathcal U_{\theta_{t-s}\w}(s)\bigl(\mathcal U_{\theta_{-s}\w}(t)\,\be\bigr)\bigr)\leq C_1(J,\theta_{t}\omega)\,e^{J s}\,,
\]
from which it follows that $\mathcal U_\w(t)\,w(\w)=\lim_{s\to\infty} \mathcal U_{\theta_{t-s}\w}(s)\,\be = w(\theta_t\w)$, as stated.

\smallskip
(2) First notice that if $t\leq 0$ and $\w\in\bar\Omega_1$
\begin{equation}
\label{tle0}
U_{\theta_t\w}(-t)\,w_\w(t) = \frac{U_{\theta_t\w}(-t)\,w(\theta_t\w)} {\n{U_{\theta_t\w}(-t)\,w(\theta_t\w)}} = w(\w)\,.
\end{equation}
Now we check that $w_\w$ is an entire orbit, i.e.,  $w_\w(s + t) = U_{\theta_{s}\w}(t)\, w_\w(s)$  for each $s \in \R$ and $t\geq 0$. If $t\geq 0$ and $s\geq 0$ it is immediate. If $t\geq 0$, $s\leq 0$ and $t+s\geq 0$, from~\eqref{eq-cocycle} and~\eqref{tle0}
\begin{equation*}
U_{\theta_{s}\w}(t) \, w_\w(s) = U_\w(t+s)(U_{\theta_{-s}\w}(-s) \, w_\w(s)) = U_\w(t+s)\,w(\w) = w_\w(t+s) \, .
\end{equation*}
Next, since $\theta_{t+s}\w=\theta_t(\theta_s\w)$, if $t\geq 0$, $s\leq 0$ and $t+s\leq 0$, from
\[w_\w(t+s)=\frac{w(\theta_{t+s}\w)}{\n{U_{\theta_{t+s}\w}(-t-s)\,w(\theta_{t+s}\w)}}\,,\quad w(\theta_{t+s}\w)=\frac{U_{\theta_s\w}(t)\,w(\theta_s)} {\n{U_{\theta_s\w}(t)\,w(\theta_s)}}\,,\]
and   $U_{\theta_{t+s}\w}(-t-s)\circ U_{\theta_s\w}(t)=U_{\theta_s\w}(-s)$, we deduce that
\[w_\w(t+s)=\frac{U_{\theta_s\w}(t)\,w(\theta_s\w)}{\n{U_{\theta_s\w}(-s)\,w(\theta_s\w)}}= U_{\theta_s\w}(t)\,\w_\w(s)\,,\]
and  $w_\w$ is an entire positive orbit, as claimed.\par\smallskip
Finally we check the uniqueness. Let $v_\w$ be another entire positive orbit of $U_\w$. From $v_\w(s + t) = U_{\theta_{s}\w}(t)\, v_\w(s)$ we deduce that $v_\w(s)=U_{\theta_{s-1}\w}(1)\,v_\w(s-1)$, and since $v_\w(s)\in X^+\setminus\{0\}$ for each $s\in \R$, the focusing condition (A3) yields $v_\w(s)\in C_\be$ for each $s\in\R$. Therefore,
\[u=\frac{v_\w(-s+t)}{\n{v_\w(-s+t)}}\in \Sigma\,,\]
and from Proposition~\ref{prop5.4}(1) and the definition of $w(\w)$ it follows that
\[0=\lim_{s\to\infty} d\bigl(\mathcal U_{\theta_{-s}(\theta_t\w)}(s)\,u,w(\theta_t\w)\bigr)=
\lim_{s\to\infty} d\left(\frac{v_\w(t)}{\n{v_\w(t)}},w(\theta_t\w)\right)\]
for each $t\in\R$ and $\w\in\bar\Omega_1$.  Since, by Lemma~\ref{lm:Sigma-complete}, $d$ is a metric on $\Sigma$, $v_\w(t)=\n{v_\w(t)}\,w(\theta_t\w)$ for each $t \in \R$.  From this we deduce that $v_\w(0)=\n{v_\w(0)}\,w(\w)$, and hence
\[v_\w(t)=U_\w(t)\,v_\w(0)= \n{v_\w(0)}\,U_\w(t)\,w(\w)=\n{v_{\w}(0)}\,w_\w(t)\,,\]
i.e., they coincide up to multiplication by a positive scalar, as stated.\par\smallskip
(3) Since $\ln\n{w_\w(1)}= \ln \n{U_\w(1)\,w(\w)}$, from assumption (A1) we deduce that the map
$[ \, \Omega \ni \omega\mapsto \lnplus \n{w_\w(1)}\,] \in L_1\OFP$. Moreover,
\[\ln\n{w_\w(t+s)}=\ln\n{w_{\theta_s\w}(t)}+\ln\n{w_\w(s)}\quad \textrm{ for each } t,\,s\in \R.\]
Therefore, the application of Birkhoff ergodic theorem to $\left(\OFP,(\theta_n)_{n\in\Z}\right)$ and $\ln\n{w_\w(1)}$ provides a subset $\Omega_1'\subset \bar\Omega_1$ with $\theta_1(\Omega_1')=\Omega_1'$ and $\PP(\Omega_1')=1$, and an invariant measurable map $g$ with $g^+\in L_1\OFP$ such that
\[\lim_{n\to\pm\infty}\frac{1}{n}\sum_{k=1}^n\ln\n{w_{\theta_k\w}(1)}=\lim_{n\to\pm\infty} \frac{1}{n}\ln\n{w_\w(n)} =g(\w)\]
for each $\w\in\Omega_1'$ and
\begin{equation}\label{birkdisc}
\int_\Omega\ln\n{w_{\omega'}(1)}\, d\PP(\omega') = \Int_\Omega g(\omega') \, d\PP(\omega')\, .
\end{equation}
\par\smallskip
Next we take the invariant set  $\widetilde\Omega_1 = \cup_{s\in[0,1]}\theta_s\Omega_1'$. From assumption~(A1), as in Lemma~3.4 of Lian and Lu~\cite{Lian-Lu}, we check that if $\w\in\widetilde\Omega_1$ with $\w=\theta_{s_\w}\widetilde\w$ for some $s_\w\in[0,1)$ and  $\widetilde\w\in\Omega_1'$
\begin{align*}
\limsup_{t\to\pm\infty}\frac{1}{t}\ln\n{w_\w(t)}&\leq \lim_{n\to\pm\infty} \frac{1}{n}\ln\n{w_{\widetilde\w}(n)}=g(\widetilde \w)\,, \\
\liminf_{t\to\pm\infty}\frac{1}{t}\ln\n{w_\w(t)}&\geq \lim_{n\to\pm\infty}
\frac{1}{n}\ln\n{w_{\widetilde\w}(n)}=g(\widetilde \w)\,,
\end{align*}
that is, there exists the limit and coincide with $g(\widetilde\w)$,
\[\lim_{t\to\pm\infty}\frac{1}{t}\ln\n{w_\w(t)} = g(\widetilde \w)\,. \]
Finally, since the function on the left is invariant and hence constant a.e., from~\eqref{birkdisc} we conclude that
\begin{equation*}
\lim_{t\to\pm\infty}\frac{1}{t}\ln\n{w_\w(t)} = \int_\Omega\ln\n{w_{\omega'}(1)} \, d\PP(\omega')
\end{equation*}
for each $\w \in \widetilde\Omega_1$, as stated.
\par\smallskip
(4) The focusing condition (A3) yields $\beta(\w,u)\,\be\leq U_\w(1)\,u\leq \varkappa(\w)\,\beta(\w,u)\,\be$ and, together with
$\beta(\w,w(\w))\,\be\leq U_\w(1)\,w(\w)\leq \varkappa(\w)\,\beta(\w,w(\w))\,\be\,$, we deduce that
\[
0< U_\w(1)\,u\leq \frac{\varkappa(\w)\,\beta(\w,u)}{\beta(\w,w(\w))}\,U_\w(1)\,w(\w)\leq \varkappa^2(\w)\, U_\w(1)\,u\,.
\]
The monotonicity assumption (A2) and $U_\w(t)=U_{\theta_1 \w}(t-1)\circ U_\w(1)$ implies that
\[
0< U_\w(t)\,u\leq \frac{\varkappa(\w)\,\beta(\w,u)}{\beta(\w,w(\w))}\,U_\w(t)\,w(\w)\leq \varkappa^2(\w)\, U_\w(t)\,u\quad \text{ for } t\geq 1\,,\]
and the normal character of the positive cone $X^+$ provides
\[
\n{U_\w(t)\,u}\leq \frac{\varkappa(\w)\,\beta(\w,u)}{\beta(\w,w(\w))}\,\n{U_\w(t)\,w(\w)}\leq \varkappa^2(\w)\, \n{U_\w(t)\,u}\quad \text{ for } t\geq 1\,,
\]
from which we conclude that
\begin{equation*}
\lim_{t\to\infty} \frac{1}{t}\ln\n{U_\w(t)\,u} = \lim_{t\to\infty} \frac{1}{t}\ln\n{U_\w(t)\,w(\w)} = \bar\lambda_1\,,
\end{equation*}
as claimed.\par\smallskip
(5) Since we are assuming that the cone $X^+$ is reproducing, i.e., $X=X^+-X^+$, we can decompose
$u\in X$  as $u=u^+-v^+$ for some $u^+$ and $v^+\in X^+$. Thus, denoting $|u|=u^++v^+\in X^+$, we have $-|u|\leq u\leq |u|$, and again the monotonicity assumption (A2) yields
\[-U_\w(t)\,|u| \leq U_\w(t)\,u\leq U_\w(t)\,|u|\,,\quad t\geq 0\,.\] Therefore, we deduce that $0\leq U_\w(t)\,|u|+U_\w(t)\,u\leq 2\,U_\w(t)\,|u|$, hence, the normal character of the cone provides $\n{U_\w(t)\,u}\leq 3\,\n{U_\w(t)\,|u|}$, and inequality~\eqref{ineqLyap} follows from relation~\eqref{eqLyap}.
\end{proof}
\begin{nota}
\label{rm:compar-with-MiSh_1}
When one replaces (A3) with (A3-O), part~(4) of Theorem~\ref{teor3.6} takes the form
\begin{enumerate}
\item[(4)]
{\em for each $\w \in \widetilde\Omega_1$ and $u \in X^+ \setminus \{0\}$
\begin{equation}
\label{eqLyap-1}
\lim_{t\to\infty} \frac{1}{t} \ln\n{U_\w(t) \, u} = \bar\lambda_1\,.
\end{equation}
}
\end{enumerate}
Consequently, Theorem~\ref{teor3.6} is an extension of Theorem~2.2 in~\cite{MiShPart3}.
\end{nota}
\begin{nota}\label{lyapunovexponent}
Note that as in Proposition~2.2 of Mierczy\'nski and Shen~\cite{MiShPart3}, from  the existence of the family of generalized Floquet subspaces and the generalized principal Lyapunov exponent $\bar\lambda_1$ obtained in the previous Theorem~\ref{teor3.6}, it follows that
\begin{equation*}
\lim_{t\to\infty}\frac{\ln\n{U_\w(t)}}{t} = \bar\lambda_1
\end{equation*}
for each $\w\in\widetilde\Omega_1$.
\end{nota}
The following theorem provides a counterpart of Theorem~\ref{teor3.6} for the dual system. As in~\cite{MiShPart1}, we define, for $t \geq 0$ and $\w \in \Omega$
\begin{equation}\label{defiw*}
\begin{array}{lccc}
\mathcal U_\w^*(t)\colon & S_1((X^*)^+) &\longrightarrow & S_1((X^*)^+) \\[.1cm]
& u^* & \mapsto & \displaystyle{\frac{U_\w^*(t)\,u^*}{\n{U_\w^*(t)\,u^*}}}\,.
\end{array}
\end{equation}
By Lemma~\ref{lm:dual-injective}, the above mappings are well defined.  Further, it follows from~\eqref{eq-cocycle-dual} that
\begin{equation}
\label{cocycle2-dual}
\mathcal{U}^{*}_{\theta_{-t}\w}(s) \circ \mathcal{U}^{*}_{\w}(t) = \mathcal{U}^{*}_{\w}(t+s) \,, \quad \textrm{for each } \, \w \in \Omega \textrm{ and } s, \, t \in \R^{+} \,.
\end{equation}
\begin{lema}
\label{lm:auxiliary}
Assume~\textup{(A3)*}.  For any $s \ge 1$ and any nonzero $u^{*} \in (X^{*})^{+}$ there holds $U_{\theta_s\w}^*(s) \, u^{*} \in C^{*}_{U^{*}_{\omega}(1)\, \mathbf{e}^{*}}$, where
\begin{equation}
C^{*}_{U^{*}_{\omega}(1) \mathbf{e}^{*}} := \{ v^{*} \in (X^{*})^{+} \setminus \{0\} \mid v^{*} \sim U^{*}_{\omega}(1)\, \mathbf{e}^{*} \}
\end{equation}
denotes the component of $U^{*}_{\omega}(1) \,\mathbf{e}^{*}$ in $(X^{*})^{+}$.
\end{lema}
\begin{proof}
By the cocycle property~\eqref{eq-cocycle-dual}, $U_{\theta_s\w}^*(s) \, u^{*} = U^{*}_{\omega}(1) \bigl( U^{*}_{\theta_{s - 1}\omega}(s - 1) \, u^{*} \bigr)$.  Lemma \ref{lm:dual-injective} gives that $U^{*}_{\theta_{s - 1}\omega}(s - 1) \, u^{*} \in (X^{*})^{+} \setminus \{0\}$.  Now we need to apply~\eqref{focusing*}.
\end{proof}
\smallskip
The following are counterparts of Propositions~\ref{prop5.3} and~\ref{prop5.4} for the dual system.
\begin{prop}
\label{prop5.3-dual}
Under assumption \textup{(A3)*},
\begin{itemize}
\item[(i)]
for each $\w \in \Omega$, $t \geq 2$ and $u^{*}$, $\widetilde{u}^{*} \in S_1((X^{*})^{+})$
\begin{equation*}
 d \bigl(\mathcal{U}^{*}_{\theta_{t}\w}(t) \, u^{*}, \, \mathcal{U}^{*}_{\theta_{t}\w}(t) \, \widetilde{u}^{*} \bigr) \leq 2 \, \ln\varkappa^{*}\bigl(\theta_{\lfloor t \rfloor }\w\bigr) \, q^{*} \bigl(\theta_{ \lfloor t \rfloor - 1}\w\bigr) \cdots q^{*}\bigl(\theta_{1}\w\bigr)\,,
\end{equation*}
\item[(ii)]
for each $\w\in\Omega$, $t\geq 2$ and $u^{*}$, $\widetilde{u}^{*} \in S_1((X^{*})^{+})$
\begin{equation*}
d\bigl(\mathcal{U}^{*}_\w(t) \, u^{*},\, \mathcal{U}^{*}_\w(t) \, \widetilde{u}^{*} \bigr) \leq
2 \, \ln\varkappa^{*}(\w) \, q^{*}\bigl(\theta_{-[t]+1}\w \bigr)\cdots q^{*} \bigl(\theta_{-1}\w\bigr) \,.
\end{equation*}
\end{itemize}
\end{prop}
\begin{prop}
\label{prop5.4-dual}
Under assumption \textup{(A3)*}, let $I:=\int_\Omega\ln q\, d\PP < 0$. Then, there is an invariant set $\bar\Omega^{*}_1 \subset \Omega$ with $\PP(\bar\Omega^{*}_1) = 1$ such that
\begin{itemize}
\item[(1)]
for each $I < J < 0$ and  $\w \in \bar\Omega^{*}_1$, there is a $C_3(J,\w) > 0$ such that
\begin{equation*}
d\bigl(\mathcal{U}^{*}_{\theta_{t}\w}(t) \, u, \, \mathcal{U}^{*}_{\theta_{t}\w}(t) \, \widetilde{u}^{*} \bigr)\leq C_3(J,\w) \, e^{J t}
\end{equation*}
whenever $t \geq 3$ and $u^{*}$, $\widetilde{u}^{*} \in S_1((X^{*})^{+})$,
\item[(2)]
for each $I < J < 0$ and  $\w \in \bar\Omega^{*}_1$, there is a $C_4(J,\w) > 0$ such that
\begin{equation*}
d\bigl(\mathcal{U}^{*}_\w(t) \, u^{*},\, \mathcal{U}^{*}_\w(t) \, \widetilde{u}^{*} \bigr) \leq C_4(J,\w) \, e^{J t}
\end{equation*}
whenever $t \geq 2$ and $u^{*}$, $\widetilde{u}^{*} \in S_1((X^{*})^{+})$.
\end{itemize}
\end{prop}
Proposition~\ref{prop5.4-dual}(1) ensures that for any $\w \in \bar\Omega^{*}_1$ the following exists
\begin{equation}
\label{w^*(w)}
w^*(\w) := \lim\limits_{s\to \infty} \mathcal{U}_{\theta_s\w}^*(s) \, \be^*,
\end{equation}
where the limit is taken in $d$.  Since, by Lemma~\ref{lm:auxiliary}, $\mathcal{U}_{\theta_s\w}^*(s)\,\be^*$ belongs, for any $s \ge 1$, to $C^{*}_{U^{*}_{\omega}(1) \,\mathbf{e}^{*}} \cap S_1((X^{*})^{+})$, and since, by a counterpart of Lemma~\ref{lm:Sigma-complete}, $(C^{*}_{U^{*}_{\omega}(1)\, \mathbf{e}^{*}} \cap S_1((X^{*})^{+}), d)$ is a complete metric space, $w^{*}(\w)$ belongs to $(C^{*}_{U^{*}_{\omega}(1)\, \mathbf{e}^{*}} \cap S_1((X^{*})^{+})$.  Further, it follows from a counterpart of Lemma~\ref{lm:projective-vs-norm} that the above limit can be taken in the $X^{*}$\nobreakdash-\hspace{0pt}norm.  Moreover, since the functions $[\w \mapsto \mathcal{U}^{*}_{\theta_{n}\w}(n) \, \be^*]$ are $(\mathfrak{F}, \mathfrak{B}(X^{*}))$\nobreakdash-\hspace{0pt}measurable, the function $w^{*} \colon \bar\Omega_1 \to X^{*}$ is measurable.  We want to remark that now $w^*(\w)$ does not necessarily belong to $C_{\be^*}$, because of the different focusing condition~\eqref{focusing*}.
\par\smallskip
Moreover, we claim that,  if $\bar\Omega_1$ is the invariant set of Proposition~\ref{prop5.4} and $w$ is defined by~\eqref{w(w)}, for each $\w\in\bar\Omega_1\cap\bar\Omega_1^*$
\begin{equation}\label{inequww*}
0<\langle w(\w),w^*(\w)\rangle\leq 1\,.
\end{equation}
The right inequality is immediate because they are unitary vectors. Concerning the left one, since $w(\w)\in C_\be$ there is an $\alpha>0$ such that $\alpha\,\be\leq w(\w)$ and, consequently
$\langle w(\w),w^*(\w)\rangle\geq \alpha\,\langle\be,w^*(\w)\rangle$. Next we check that $\langle\be,w^*(\w)\rangle>0$. Otherwise, from $\langle\be,w^*(\w)\rangle=0$  we would deduce that $\langle \be,U_\w^*(1)\,w^*(\w)\rangle=0$, and from relation~\eqref{focusing*} that $\langle \be,U_\w^*(1)\,\be^*\rangle=0$, that is, $\langle U_{\theta_{-1}\w}(1)\,\be,\be^*\rangle=0$. This contradicts that $U_{\theta_{-1}\w}(1)\,\be\in C_\be$ (see~\eqref{focusing}) and $\langle \be,\be^*\rangle=1$, and proves the assertion.
\begin{teor}\label{teor3.7}
 Under assumptions {\rm (A1)*} and {\rm (A3)*}, if $\bar\Omega_1^*$ is the invariant set of~\eqref{w^*(w)}, there is  an $(\mathfrak{F},\mathfrak{B}(X^*))$-measurable function
\[
w^*\colon\bar\Omega_1^*\to S_1((X^*)^+),\;\w\mapsto w^*(\w)
\]
 defined by~\eqref{w^*(w)} such that
\begin{itemize}
\item[(1)]
for each $\w \in \bar\Omega_1^*$ and $t \geq 0$,
\begin{equation}\label{prop_w^*}
w^*(\theta_{-t}\w)=\frac{U_\w^*(t)\,w^*(\w)}{\n{U_\w^*(t)\,w^*(\w)}}\,;
\end{equation}
\item[(2)]
for each $\w\in\bar\Omega_1^*$, the map $w_\w^*\colon \R\to X^+$ defined by
\begin{equation}
\label{defiw*w(t)}
w_\w^*(t)=\begin{cases} \displaystyle\frac{w^*(\theta_{-t}\w)} {\n{U^{*}_{\theta_{-t}\w}(-t) \, w^*(\theta_{-t}\w)}} & \quad \text{for } t < 0\,,\\[.4cm]
 \;U_\w^*(t)\,w^*(\w)& \quad \text{for } t\geq 0\,; \end{cases}
\end{equation}
is a positive entire orbit of $U_\w^*$, unique
up to multiplication by a positive scalar;
\item[(3)]
there are an invariant set $\widetilde\Omega_1^*\subset\bar\Omega_1^*$ with $\PP(\widetilde\Omega_1^*) = 1$ and  a $\bar\lambda_1^* \in [-\infty,\infty)$ such that
\begin{equation*}
\bar\lambda_1^* = \lim_{t\to\pm\infty}\frac{1}{t}\ln\n{w_\w^*(t)} = \int_\Omega\ln\n{w^{*}_{\omega'}(1)} \, d\PP(\omega')
\end{equation*}
for each $\w \in \widetilde\Omega_1^*$;
\item[(4)]
The generalized principal Lyapunov exponent $\bar\lambda_1$ obtained in {\rm Theorem~\ref{teor3.6}} coincides with $\bar\lambda_1^*$.
\end{itemize}
\end{teor}
\begin{proof}  (1)  From relation~\eqref{cocycle2-dual} and the definition of $w^{*}(\w)$  we deduce that
\begin{align*}
\mathcal{U}^{*}_\w(t) \, w(\w) & = \mathcal{U}^{*}_\w(t) \left( \lim_{s \to \infty} \mathcal{U}^{*}_{\theta_{s}\w}(s) \, \mathbf{e}^{*} \right) = \lim_{s\to\infty} \bigl( \mathcal{U}^{*}_\w(t) \circ \mathcal{U}^{*}_{\theta_{s}\w}(s) \bigr) \, \mathbf{e}^{*}
\\
& = \lim_{s\to\infty} \mathcal{U}^{*}_{\theta_{s}\w}(s+t) \, \mathbf{e}^{*} = \lim_{s\to\infty} \bigl( \mathcal{U}^{*}_{\theta_{s-t}\w}(s) \circ \mathcal U_{\theta_{s}\w}(t) \bigr) \, \mathbf{e}^{*} \,,
\end{align*}
for each $\w \in \bar\Omega^{*}_1$ and $t \geq 0$. Moreover, from Proposition~\ref{prop5.4-dual}(1)
\begin{equation*}
d\bigl(\mathcal{U}^{*}_{\theta_{s-t}\w}(s) \, \mathbf{e}^{*}, \, \mathcal{U}^{*}_{\theta_{s-t}\w}(s) \bigl( \mathcal{U}^{*}_{\theta_{s}\w}(t) \, \mathbf{e}^{*} \bigr) \bigr) \leq C_3(J, \theta_{-t}\omega) \, e^{Js}\,,
\end{equation*}
from which it follows that $\mathcal{U}^{*}_\w(t) \, w^{*}(\w) = \lim_{s\to\infty} \mathcal{U}^{*}_{\theta_{s-t}\w}(s) \, \mathbf{e}^{*} = w^{*}(\theta_{-t}\w)$, as stated.
\par
\smallskip
(2) The fact that $w^{*}_{\w}$ is a positive entire orbit follows along the lines of the proof of Theorem~\ref{teor3.6}(2).
\par
\smallskip
We check the uniqueness. Let $v^{*}_\w$ be another entire positive orbit of $U^{*}_\w$. It follows from Proposition~\ref{prop5.4-dual}(1)  and the definition of $w^*$ that
\begin{equation*}
0 = \lim_{s \to \infty} d\left(\mathcal{U}^{*}_{\theta_{s}(\theta_t\w)}(s) \, \frac{v^{*}_\w(s+t)}{\n{v^{*}_\w(s+t)}}, w^{*}(\theta_t\w) \right) =
\lim_{s\to\infty} d\left(\frac{v^{*}_\w(t)}{\n{v^{*}_\w(t)}}, w^{*}(\theta_t\w)\right)
\end{equation*}
for each $t\in\R$ and $\w\in\bar\Omega^{*}_1$.  By Lemma~\ref{lm:auxiliary},
\begin{equation*}
\mathcal{U}^{*}_{\theta_{s}(\theta_t\w)}(s) \, \frac{v^{*}_\w(s+t)}{\n{v^{*}_\w(s+t)}} \in C^{*}_{U^{*}_{\theta_{t - 1}\omega}(1)\, \mathbf{e}^{*}}
\end{equation*}
for $s \ge 1$, and by the counterpart of Lemma~\ref{lm:Sigma-complete}, $(C^{*}_{U^{*}_{\theta_{t - 1}\omega}(1) \mathbf{e}^{*}}, d)$ is a complete metric space.  Consequently, $v^{*}_\w(t)=\n{v^{*}_\w(t)}\,w^{*}(\theta_t\w)$ for each $t \in \R$.  From this we deduce that $v^{*}_\w(0) = \n{v^{*}_\w(0)} \, w^{*}(\w)$, and hence
\begin{equation*}
v^{*}_\w(t) = U^{*}_\w(t) \, v^{*}_\w(0) = \n{v^{*}_\w(0)} \, U^{*}_\w(t) \, w^{*}(\w) = \n{v^{*}_{\w}(0)} \, w^{*}_\w(t)\,,
\end{equation*}
i.e., they coincide up to multiplication by a positive scalar, as stated.
\par
\smallskip
(3) The proof goes along the lines of the proof of Theorem~\ref{teor3.6}(3).
\par
\smallskip
(4) Let $\w \in \widetilde{\Omega}_1 \cap \widetilde{\Omega}_1^*$.  First note that, with the help of~\eqref{eq:definition-of-w} and~\eqref{defiw*w(t)},
\begin{align}
\label{lambda}
\bar\lambda_1 & = \lim_{t\to\infty}\frac{1}{t}\ln\n{w_\w(t)} = \lim_{t\to\infty} \frac{1}{t}\ln\n{U_\w(t)\,w(\w)}\,,
\\
\label{lambda-1}
\bar{\lambda}_1 &  = \lim_{t \to \infty} \frac{1}{-t} \ln\n{w_{\w}(-t)} = \lim_{t \to \infty} \frac{1}{t}\ln\n{U_{\theta_{-t}\omega}(t)\,w(\theta_{-t}\omega)}\,,
\\
\label{lambda-2}
\bar{\lambda}_1^* &  = \lim_{t \to \infty} \frac{1}{-t} \ln\n{w_\w^*(-t)} = \lim_{t\to\infty} \frac{1}{t} \ln\n{U_{\theta_t\w}^*(t)\,w^*(\theta_t\w)}\,,
\\
\label{lambda-3}
\bar{\lambda}_1^* &  = \lim_{t \to \infty} \frac{1}{t} \ln\n{w_\w^*(t)} = \lim_{t\to\infty} \frac{1}{t} \ln\n{U_{\w}^*(t)\,w^*(\w)}\,.
\end{align}
Moreover, from relations~\eqref{dual-definition}, \eqref{prop_w} and~\eqref{prop_w^*} we obtain
\begin{align}\label{UU*}
\n{U_w(t)\,w(\w)}\,\langle w(\theta_t\w),w^*(\theta_t\w)\rangle & = \langle U_\w(t)\,w(\w),w^*(\theta_t\w) \rangle \nonumber
\\
& = \langle w(\w),U_{\theta_t\w}^{*}(t)\,w^*(\theta_t\w) \rangle
\\
& = \n{U_{\theta_tw}^*(t)\,w^*(\theta_t \w)}\,\langle w(\w),w^*(\w)\rangle\,,\nonumber
\end{align}
which together with $\lim_{t\to\infty}(1/t)\ln \langle w(\w),w^*(\w)\rangle=0$, \eqref{lambda-3}, \eqref{lambda} and~\eqref{inequww*} provides that
\begin{equation}\label{desilambdas}
\begin{split}
\bar\lambda_1^* &=\lim_{t\to\infty}\left(\frac{1}{t}\ln \n{U_w(t)\,w(\w)}+\frac{1}{t}\ln \langle w(\theta_t\w),w^*(\theta_t\w)\rangle \right)\\
&\leq \lim_{t\to\infty}\frac{1}{t}\ln\n{U_\w(t)\,w(\w)} +\limsup_{t\to\infty} \frac{1}{t}\ln \langle w(\theta_t\w),w^*(\theta_t\w)\rangle\leq \bar\lambda_1\,.
\end{split}
\end{equation}
Analogously, using that
\begin{equation*}
\n{U_w^*(t)\,w^*(\w)}\,\langle w(\theta_{-t}\w),w^*(\theta_{-t}\w)\rangle = \n{U_{\theta_{-t}w}(t)\,w(\theta_{-t} \w)}\,\langle w(\w),w^*(\w)\rangle\,,
\end{equation*}
which follows by changing $\w$ to $\theta_{-t}\w$ in~\eqref{UU*}, we deduce, with the help of~\eqref{lambda-1} and~\eqref{lambda-3}, that $\bar\lambda_1 \leq \bar\lambda_1^*$, which finishes the proof.
\end{proof}
\begin{nota}
In view of Proposition~\ref{A3->A3*}, (A1)* and (A3)* can be replaced in Theorem~\ref{teor3.7} by (A1) and (A3).  Therefore, Theorem~\ref{teor3.7} is new even in the case of assumption~(A3-O).
\end{nota}
\section{Generalized exponential separation}\label{sec-GES}
As stated before, $X$ is an ordered separable Banach space such that $X^{*}$ is separable, with positive cone $X^+$ normal and reproducing, and recall that, once (A3) is assumed, assumptions (A2), (A2)* and (A3)* hold.
\par
\smallskip
In this section we will prove the existence of a generalized exponential separation of type II, now introduced and important for cases in which the previous concept of generalized exponential separation does not apply, as \mlsps s  induced by delay differential equations.
\begin{defi}
\label{def:exponential-separation}
The \mlsps  $\;\Phi = ((U_\w(t)), \allowbreak (\theta_t))$ is said to admit a {\em generalized exponential separation of type} II if there are a family of generalized principal Floquet subspaces $\{\widetilde{E}(\w)\}_{\w \in \widetilde{\Omega}}$,
and a family of one\nbd-codimensional closed vector subspaces $\{\widetilde{F}(\w)\}_{\w \in \widetilde{\Omega}}$ of $X$,  satisfying
\begin{itemize}
\item[{\rm (i)}]
$\widetilde{F}(\w)\cap X^+=\{u\in X^+ \mid U_\w(1)\,u=0\}$,
\item[{\rm (ii)}]
$X=\widetilde{E}(\w)\oplus \widetilde{F}(\w)$ for any $\w\in\widetilde{\Omega}$, where the decomposition is invariant, and the family of projections associated with the decomposition is strongly measurable and tempered,
\item[{\rm (iii)}]
there exists $\widetilde{\sigma}\in (0,\infty]$ such that
\[ \lim_{t\to\infty}\frac{1}{t}\ln \frac{\n{U_\w(t)|_{\widetilde{F}(\w)}}}{\n{U_\w(t)\,w(\w)}}=-\widetilde{\sigma}
\]
for each $\w\in\widetilde{\Omega}$.
\end{itemize}
We say that $\{\widetilde{E}(\cdot),\widetilde{F}(\cdot),\widetilde{\sigma}\}$ {\em generates a generalized exponential separation of type}~II.
\end{defi}
Note that the only difference with the definition of generalized exponential separation given in~\cite{MiShPart1} is that in this case $\widetilde{F}(\w)$  contains those positive vectors $u>0$ for which $U_\w(1)\,u=0$ because $U_\w(1)$ is not assumed to be injective.\par\smallskip
Next we consider $\widetilde{\Omega}_1$ and $\widetilde\Omega_1^*$, the invariant sets of Theorem~\ref{teor3.6} and~\ref{teor3.7}, and we define for each $\w\in \widetilde{\Omega}_1 \cap \widetilde\Omega_1^*$
\[F_1(\w) = \{u \in X \mid \langle u, w^*(\w) \rangle = 0 \}\,.\]
From~\eqref{inequww*}, each $u \in X$  can be decomposed as $u=\alpha\, w(\w)+u-\alpha\, w(\w)$ with \[\alpha=\frac{\langle u,w^*(\w)\rangle}{\langle w(\w),w^*(\w)\rangle}\,,\]
and hence,
\[ X=E_1(\w)\oplus F_1(\w)\,,\]
where, as denoted in Theorem~\ref{teor3.6}, $E_1(\w)=\spanned\{w(\w)\}$. As a consequence, the following result holds.
\begin{lema}\label{lema5.6}
The family $\{P(\w)\}_{\w\in \widetilde{\Omega}_1 \cap \widetilde\Omega_1^*}$ of projections associated with the decomposition $E_1(\w)\oplus F_1(\w)=X$ is given by the formula
\begin{equation}\label{defiPw}
P(\w)\,u=u - \frac{\langle u,w^*(\w)\rangle}{\langle w(\w),w^*(\w)\rangle}\,w(\w)\,,\quad \w \in \widetilde{\Omega}_1 \cap \widetilde\Omega_1^*\,.
\end{equation}
\end{lema}
\par
\smallskip
We omit the proof of the next result, based in Lemma 5.10 of~\cite{MiShPart1} with the corresponding modifications due to the different definition of the maps~\eqref{defiw} and~\eqref{w(w)}. See Definition~\ref{OscilRatio} for the oscillation of two vectors.
\begin{prop}
\label{prop5.10}
Under assumptions {\rm (A1)} and {\rm (A3)}, there exists an invariant subset $\widetilde \Omega_2 \subset \widetilde{\Omega}_1 \cap \widetilde\Omega_1^*$ of full measure $\PP(\widetilde\Omega_2) = 1$, with the property that for each $0>J >\int_\Omega\ln p\,d\PP$
and each $\w \in \widetilde \Omega_2$, there is a $C_{5}(\w,J)>0$ such that
\[
{\textit{osc}}\left( \frac{U_\w(t)\,u}{\n{w_\w(t)}}/w(\theta_t\w)\right)\leq C_{5}(J,\w) \, e^{Jt}
\]
whenever $u \in \Sigma = C_\be\cap S_1(X^+)$ and $t\geq 1$.
\end{prop}
\begin{lema}
\label{limit0}
Under assumptions {\rm (A1)} and {\rm (A3)}, let $\bar\lambda_1$ be the generalized principal Lyapunov exponent of {\rm Theorem~\ref{teor3.6}}. If $\bar\lambda_1 > -\infty$ then
\[\lim_{t\to\pm\infty} \frac{1}{t}\ln\langle w(\theta_t\w), w^*(\theta_t\w)\rangle =0\]
for each $\w \in \widetilde{\Omega}_1 \cap \widetilde{\Omega}_1^*$.
\end{lema}
\begin{proof}
We prove the result for $t \to \infty$, the other limit is completely analogous.  Since $\bar\lambda_1 = \bar\lambda_1^*$, from inequality~\eqref{desilambdas} we deduce that
\[
\limsup_{t\to\infty} \frac{1}{t}\ln \langle w(\theta_t\w),w^*(\theta_t\w)\rangle=0\,.
\]
Assume now on the contrary to the assertion of the lemma that there is a sequence $t_n\uparrow\infty$ such that
\[
\lim_{n\to\infty} \frac{1}{t_n}\ln \langle w(\theta_{t_n}\w),w^*(\theta_{t_n}\w)\rangle = a < 0\,.
\]
Again from~\eqref{desilambdas} we would obtain $\bar\lambda_1^*=\bar\lambda_1+a$, a contradiction.
\end{proof}
We include part of the proof of the next result, although similar to the proof of Proposition~5.11 of~\cite{MiShPart1}, to remark the differences derived from the new assumption (A3) and the fact that we do not have a Banach lattice but an ordered separable Banach space with positive cone $X^+$ normal and reproducing.
\begin{prop}
\label{prop5.11}
 Under assumptions {\rm (A1)} and {\rm (A3)}, let $\widetilde\Omega_2$ be the invariant subset of {\rm Proposition~\ref{prop5.10}}. If $\bar\lambda_1 > -\infty$, there exists a $\bar\sigma_2 > 0$ such that
\[
\limsup_{t\to\infty}\frac{1}{t}\ln\sup_{u\in X,\,\n{u}=1}\left\{\,\left\|\frac{U_\w(t)\,u}{\n{w_\w(t)}}-\frac{\langle u,w^*(\w)\rangle}{\langle w(\w),w^*(\w)\rangle}\,w(\theta_t\w)\right\|\,\right\}\leq -\bar\sigma_2
\]
for each $\w\in\widetilde\Omega_2$.
\end{prop}
\begin{proof}
Denote $\widetilde U_\w(t)\,u:=U_\w(t)\,u/\n{w_\w(t)}$. As in Proposition~5.11 of~\cite{MiShPart1}, from Proposition~\ref{prop5.10} we prove that
$m(\widetilde U_\w(t)\,u/w(\theta_t\w))$ and $M(\widetilde U_\w(t)\,u/w(\theta_t\w))$ both converge
to a common limit denoted by $\mu(u,\w)$ and
 \begin{align*}
 \mu(u,\w)-m(\widetilde U_\w(t)\,u/w(\theta_t\w))&\leq C_{5}(J,\w)\,e^{Jt}\,,\\
 M(\widetilde U_\w(t)\,u/w(\theta_t\w))-\mu(u,\w)& \leq C_{5}(J,\w)\,e^{Jt}
 \end{align*}
 for each $t\geq 1$ and $u\in \Sigma=C_\be\cap S_1(X^+)$. Moreover, since
 \begin{align*}
 \bigl(m(\widetilde U_\w(t)\,u/w(\theta_t\w))-\mu(u,\w)\bigr)\,w(\theta_t\w)&\leq \widetilde U_\w(t)\,u-\mu(t,\w)\,w(\theta_t\w) \\
 &\leq \bigl(M(\widetilde U_\w(t)\,u/w(\theta_t\w))-\mu(u,\w)\bigr)\,w(\theta_t\w)\,,
 \end{align*}
 and $X^+$ is normal  we deduce that
 \begin{equation}\label{expodecay}
   \left\|\,\frac{U_\w(t)\,u}{\n{w_\w(t)}}-\mu(u,\w)\,w(\theta_t\w)\,\right\|\leq 3\, C_{5}(J,\w)\,e^{Jt}
 \end{equation}
 for each $t\geq 1$ and $u \in \Sigma = C_\be \cap S_1(X^+)$.
 Next we fix $u\in \Sigma$ and $\w\in \widetilde\Omega_2$. It is not hard to check, as in Proposition~5.11 of~\cite{MiShPart1}, that
 \begin{multline*}
 \left\langle \frac{U_\w(t)\,u}{\n{w_\w(t)}}-\mu(u,\w)\, w(\theta_t\w),w^*(\theta_t\w)\right\rangle \\ =\left(\frac{\langle u, w^*(\w)\rangle}{\langle w(\w),\w^*(\w)\rangle}-\mu(u,\w)\right) \langle w(\theta_t\w), w^*(\theta_t\w)\rangle\,,
 \end{multline*}
 which together with the exponential decay~\eqref{expodecay} and Lemma~\ref{limit0} provides
 \[\mu(u,\w)=\frac{\langle u, w^*(\w)\rangle}{\langle w(\w),\w^*(\w)\rangle}\,.\]
 Therefore, we have proved that
\begin{equation}\label{desiCe}
 \left\|\,\frac{U_\w(t)\,u}{\n{w_\w(t)}}- \frac{\langle u, w^*(\w)\rangle}{\langle w(\w),\w^*(\w)\rangle} \,w(\theta_t\w)\,\right\|\leq 3\, C_{5}(J,\w)\,\n{u}\,e^{Jt}
 \end{equation}
for each $t\geq 1$, $\w\in\widetilde\Omega_2$ and $u \in C_\be$. Now let $u \in X^{+}$. If $U_\w(1)\,u = 0$ the left-hand side of the previous inequality vanishes for each $t\geq 1$ because $U_\w(t)\,u=0$ and from~\eqref{prop_w^*} and~\eqref{dual-definition} we deduce that
\[\langle u,w^*(\w)\rangle=\left\langle u, \frac{U^*_{\theta_1\w}(1)\,w^*(\theta_1\w)}{\n{U^*_{\theta_1\w}(1)\,w^*(\theta_1\w)}} \right\rangle=\left\langle U_\w(1)\,u, \frac{w^*(\theta_1\w)}{\n{U^*_{\theta_1\w}(1)\,w^*(\theta_1\w)}}\right\rangle=0\,.\]
Next,  a straightforward computation shows that, if $\widetilde u=U_\w(1)\,u$ we have
\begin{equation}\label{desitm2}
\frac{U_\w(t)\,u}{\n{w_\w(t)}}-\mu(u,\w)\, w(\theta_t\w)=\frac{1}{\n{w_\w(1)}}\left( \frac{U_{\theta_1\w}(t-1)\,\widetilde u}{\n{w_{\theta_1\w}(t-1)}}-\mu(\widetilde u,\theta_1\w)\, w(\theta_{t-1}\theta_1\w)\right)
\end{equation}
for each $t\geq 2$. Therefore,  if $u\in X^+$ and $U_\w(1)\,u\neq 0$, from condition (A3), i.e.~\eqref{focusing}, we deduce that $\widetilde u=U_\w(1)\,u\in C_\be$, and consequently from~\eqref{desiCe}
\[ \left\|\,\frac{U_\w(t)\,u}{\n{w_\w(t)}}- \frac{\langle u, w^*(\w)\rangle}{\langle w(\w),\w^*(\w)\rangle} \,w(\theta_t\w)\,\right\|\leq \frac{3}{\n{w_\w(1)}}\, C_{5}(J,\theta_1\w)\,\n{U_\w(1)}\n{u}\,e^{J(t-1)}\]
for each $t\geq 2$.
\par
\smallskip
Finally, since $X^+$ is reproducing, i.e. $X = X^+ - X^+$, there is an $\alpha>0$ such that for each $u\in X$ there are $u^+,v^+ \in X^+$, not necessarily unique, such that $u = u^+ - v^+$,  $\n{u^+} \leq \alpha\n{u}$ and $\n{v^+} \leq \alpha\n{u}$ (see~\cite[Thm. 2.2]{AbAlBu}).  If we apply the previous inequalities to the decomposition $u^+$ and $v^+ \in X^+$ for $u\in X$ with $\n{u} = 1$ we get
\begin{equation}
\label{b(w,J)}
\left\|\,\frac{U_\w(t)\,u}{\n{w_\w(t)}}- \frac{\langle u, w^*(\w)\rangle}{\langle w(\w),\w^*(\w)\rangle} \,w(\theta_t\w)\,\right\|\leq  b(\w,J)\,e^{Jt}
\end{equation}
for each $t\geq 2$ and $\w\in\widetilde \Omega_2$, where $b(\w,J) = 6 \, \alpha \, C_{5}(J,\theta_1\w) \, e^{-J} \, \n{U_\w(1)}/\n{w_\w(1)}$, which finishes the proof.
\end{proof}
The next theorem shows the existence of a generalized exponential separation of type II. We maintain the notation of the previous results.
\begin{teor}
\label{thm:exponential-separation}
Under assumptions {\rm (A1)} and {\rm (A3)}, let $\bar\lambda_1$ be the generalized principal Lyapunov exponent of {\rm Theorem~\ref{teor3.6}} and assume that $\bar\lambda_1 > -\infty$.  Then there is an invariant set $\widetilde\Omega_0$ of full measure $\PP(\widetilde\Omega_0)=1$ such that
\begin{itemize}
\item[(1)]
The family $\{ P(\w)\}_{\w\in\widetilde\Omega_0}$ of projections associated with invariant decomposition $ E_1(\w)\oplus  F_1(\w)=X$ is strongly measurable and tempered.
\item[(2)]
$F_1(\w)\cap X^+=\{u\in X^+ \mid U_\w(1)\,u=0\}$ for any $\w\in\widetilde\Omega_0$.
\item[(3)]
For any $\w\in\widetilde\Omega_0$ and $u\in X\setminus F_1(\w)$ with $U_\w(1)\,u\neq 0$  there holds
 \[ \lim_{t\to\infty}\frac{1}{t}\ln\n{U_\w(t)}=\lim_{t\to\infty}\frac{1}{t}\ln \n{U_\w(t)\,u}=\bar\lambda_1\,.\]
 \item[(4)] There exists $\widetilde\sigma\in(0,\infty]$ and $\bar\lambda_2=\bar\lambda_1-\widetilde\sigma$ such that
\[ \lim_{t\to\infty}\frac{1}{t}\ln \frac{\n{U_\w(t)|_{F_1(\w)}}}{\n{U_\w(t)\,w(\w)}}=-\widetilde{\sigma}
\]
and
\[\lim_{t\to\infty}\frac{1}{t}\ln\n{U_\w(t)|_{F_1(\w)}}=\bar\lambda_2\]
for each $\w\in\widetilde\Omega_0$, that is, $\Phi$ admits a generalized exponential separation of type {\rm II}.
\end{itemize}
\end{teor}
\begin{proof}
(1) The strong measurability follows from~\eqref{defiPw} and the measurability of $w$ and $w^*$.
Next we show that it is a tempered family, i.e.~\eqref{tempered} holds. Let $\widetilde{\Omega}_1$ and $\widetilde\Omega_1^*$ be the invariant sets of Theorem~\ref{teor3.6} and~\ref{teor3.7}.
For $\w \in \widetilde{\Omega}_1 \cap \widetilde\Omega_1^*$ we define $\widetilde P(\w)= \mathrm{Id}_{X}- P(\w)$, that is,
\[
\widetilde P(\w)\,u=\frac{\langle u,w^*(\w)\rangle}{\langle w(\w),w^*(\w)\rangle}\,w(\w)\,,\quad u\in X\,.
\]
Therefore,  $1\leq\n{\widetilde P(\w)}\leq 1/\langle w(\w),w^*(\w)\rangle$  and we deduce that
$0\leq \ln \n{\widetilde P(\w)}\leq -\ln\langle w(\w),w^*(\w)\rangle$. Since $\bar\lambda_1\neq-\infty$, Lemma~\ref{limit0} shows that
\[ \lim_{t\to\pm\infty} \frac{\ln\n{\widetilde P(\theta_t\w)}}{t}=0\quad \text{ for }\w\in \widetilde\Omega_1 \cap \widetilde\Omega_1^*\,.\]
Consequently,  $\n{P(\theta_t\w)}\leq 1+ \n{\widetilde P(\theta_t\w)} = \n{\widetilde P(\theta_t\w)}\left(1+ \n{\widetilde P(\theta_t\w)}^{-1}\right)$ and
\[\limsup_{t\to\pm\infty}\frac{\ln\n{P(\theta_t\w)}}{t}\leq \limsup_{t\to\pm\infty}\frac{\ln\left(1+\n{\widetilde P(\theta_t\w)}^{-1}\right)}{t}=0 \]
because $1\leq 1+\n{\widetilde P(\theta_t\w)}^{-1}\leq 2\,$. The inequality
\[ \liminf_{t\to\pm\infty}\frac{\ln\n{P(\theta_t\w)}}{t}\geq 0\]
follows from $\n{P(\theta_t\w)}\geq 1$, and~\eqref{tempered} holds, which finishes the proof of (1) when $\w \in \widetilde\Omega_1 \cap \widetilde\Omega_1^*$.\par\smallskip
(2) Recall that $F_1(\w)=\{u\in X\mid \langle u,w^*(\omega) \rangle = 0\}$ and from~\eqref{prop_w^*},~\eqref{dual-definition}, and the definition of $w^*_\w$ we deduce that
\begin{equation}\label{uw*(w)}
\langle u,w^*(\w)\rangle = \left\langle u, \frac{U^*_{\theta_1\w}(1)\,w^*(\theta_1\w)} {\n{U^*_{\theta_1\w}(1)\,w^*(\theta_1\w)}} \right\rangle = \left\langle U_\w(1)\,u, \frac{w^*(\theta_1\w)}{\n{w^{*}_{\omega}(1)}} \right\rangle\,.
\end{equation}
Moreover, from the focusing condition (A3) we deduce that if $u \in X^+\setminus\{0\}$ we have two options $U_\w(1)\,u=0$ and hence $u\in F_1(\w)\cap X^+$, or $U_\w(1)\,u\in C_\be$. In this case, we claim that $\langle u,w^*(\w)\rangle >0$. From~\eqref{focusing},~\eqref{uw*(w)}, and~\eqref{prop_w^*} we obtain
\begin{equation}\label{desiuw*(w)}
\langle u,w^*(\w)\rangle\geq \frac{\beta(\w,u)}{\n{w_\w^*(1)}}\,\langle \be,w^*(\theta_1\w)\rangle=\frac{\beta(\w,u)}{\n{w_\w^*(1)}}\,\left\langle \be, \frac{U^*_{\theta_2\w}(1)\,w^*(\theta_2\w)}{\n{U^*_{\theta_2\w}(1)\,w^*(\theta_2\w)}} \right\rangle\,.
\end{equation}
From the focusing condition (A3)* for $X^*$, i.e., inequality~\eqref{focusing*},
\[ U_{\theta_2\w}^*(1)\,w^*(\theta_2\w)\geq \beta^*(\theta_2\w,w^*(\theta_2\w))\,U_{\theta_2\w}^*(1)\,\be^*\,,\]
and hence, together with~\eqref{desiuw*(w)}, \eqref{prop_w^*}, and~\eqref{defiw*w(t)}  yields
\[ \langle u,w^*(\w)\rangle \geq \frac{\beta(\w,u)\,\beta^*(\theta_2\w,w^*(\theta_2\w))} {\n{w_\w^*(1)}\,\n{w^*_{\theta_2\w}(1)}}\, \langle \be, U_{\theta_2\w}^*(1)\,\be^*\rangle\,. \]
Finally, again from~\eqref{dual-definition} and~\eqref{focusing} we conclude that
 \[\langle \be, U_{\theta_2\w}^*(1)\,\be^*\rangle=\langle U_{\theta_1\w}(1)\,\be,\be^*\rangle\geq \beta(\theta_1\w,\be)\,\langle \be,\be^*\rangle>0\,,\]
and, therefore $\langle u,w^*(\w)\rangle>0$, as claimed, which finishes the proof of (2) when $\w\in \widetilde\Omega_1\cap\widetilde\Omega_1^*$.
\par\smallskip
(3) From Remark~\ref{lyapunovexponent} we know that
\[\lim_{t\to\infty}\frac{1}{t}\ln\n{U_\w(t)}=\bar\lambda_1\quad \text{ for each } \w\in\widetilde\Omega_1\,.\]
Now, let $\widetilde\Omega_2 \subset \widetilde{\Omega}_1 \cap \widetilde\Omega_1^*$ be the invariant subset of {\rm Proposition~\ref{prop5.10}}, fix $\w\in\widetilde\Omega_2$ and let $u\in X\setminus F_1(\w)$ with $U_\w(1)\,u\neq 0$. We can decompose $u=u_1+u_2$ with
\[u_2=\frac{\langle u,w^*(\w)\rangle}{\langle w(\w),w^*(\w)\rangle}\,w(\w)\,,\]
and since $\langle u,w^*(\w)\rangle\neq 0$, then $\n{u_2}>0$.
From~\eqref{prop_w} we deduce that for each $t\geq 0$
\[
U_\w(t)\,u_2=\frac{\langle u,w^*(\w)\rangle}{\langle w(\w),w^*(\w)\rangle}\,U_\w(t)\,w(\w)=
\frac{\langle u,w^*(\w)\rangle}{\langle w(\w),w^*(\w)\rangle}\n{U_\w(t)\,w(\w)}\,w(\theta_t\w)\,,
\]
and hence, $\n{U_\w(t)\,u_2}=\n{U_\w(t)\,w(\w)}\,\n{u_2}>0$. Moreover, from Proposition~\ref{prop5.11}, i.e. relation~\eqref{b(w,J)}, and~\eqref{prop_w}
\[\n{U_\w(t)\,u_1}=\biggl\|\,U_\w(t)\,u- \frac{\langle u, w^*(\w)\rangle}{\langle w(\w),\w^*(\w)\rangle} \,w(\theta_t\w)\,\biggr\|\leq  b(\w,J)\,\n{U_\w(t)\,w(\w)}\,\n{u}\,e^{Jt}\]
for each $t\geq 2$ and, consequently,
\[\n{U_\w(t)\,u}\geq \n{U_\w(t)\,u_2}-\n{U_\w(t)\,u_1}\geq \n{U_\w(t)\,w(\w)}(\n{u_2}-b(\w,J)\,\n{u}\,e^{Jt})\,.\]
Therefore,
\[
\liminf_{t\to\infty} \frac{1}{t} \, \ln{\lVert U_\w(t)\,u \rVert} \geq \liminf_{t\to\infty} \frac{1}{t} \, \ln{\lVert U_\w(t)\,w(\w) \rVert} \geq \bar\lambda_1\,,
\]
which together with Theorem~\ref{teor3.6}(5), i.e. relation~\eqref{ineqLyap}, finishes the proof of (3) when $\w\in \widetilde\Omega_2 \subset \widetilde{\Omega}_1 \cap \widetilde\Omega_1^*$.\par\smallskip
(4) We omit the proof of this part of the theorem because it follows step by step the proof of Theorem~3.8(4)~in~\cite{MiShPart1}. An invariant subset $\widetilde\Omega_0$ of full measure $\PP(\widetilde\Omega_0)=1$, and contained in $\widetilde \Omega_2 \subset \widetilde{\Omega}_1 \cap \widetilde\Omega_1^*$ where (4) holds is obtained. Consequently, all the previous results (1-3) apply for $\w \in \widetilde\Omega_0$.
\end{proof}
\begin{nota}
\label{rm:type-I}
The definition of {\em generalized exponential separation of type I\/} resembles Definition~\ref{def:exponential-separation}, the only difference being that (i) is replaced by
\begin{equation*}
\widetilde{F}(\w) \cap X^+ = \{0\}.
\end{equation*}
In~particular, generalized exponential separation of type I implies generalized exponential separation of type II.
\end{nota}
Under (A1) and (A3-O), and assuming additionally that $\bar{\lambda}_1 > - \infty$, Theorem~\ref{thm:exponential-separation} gives the existence of generalized exponential separation of type I:
\begin{teor}
\label{thm:exponential-separation-I}
Under assumptions {\rm (A1)} and {\rm (A3-O)}, let $\bar\lambda_1$ be the generalized principal Lyapunov exponent of {\rm Theorem~\ref{teor3.6}} and assume that $\bar\lambda_1 > -\infty$.  Then there is an invariant set $\widetilde\Omega_0$ of full measure $\PP(\widetilde\Omega_0)=1$ such that
\begin{itemize}
\item[(1)]
The family $\{ P(\w)\}_{\w\in\widetilde\Omega_0}$ of projections associated with invariant decomposition $ E_1(\w)\oplus  F_1(\w) = X$ is strongly measurable and tempered.
\item[(2)]
$F_1(\w)\cap X^+ = \{0\}$ for any $\w\in\widetilde\Omega_0$.
\item[(3)]
For any $\w\in\widetilde\Omega_0$ and $u\in X\setminus F_1(\w)$ there holds
\begin{equation*}
\lim_{t\to\infty}\frac{1}{t}\ln\n{U_\w(t)} = \lim_{t\to\infty}\frac{1}{t}\ln \n{U_\w(t)\,u} = \bar\lambda_1\,.
\end{equation*}
\item[(4)]
There exists $\widetilde\sigma\in(0,\infty]$ and $\bar\lambda_2 = \bar\lambda_1-\widetilde\sigma$ such that
\[
\lim_{t\to\infty}\frac{1}{t}\ln \frac{\n{U_\w(t)|_{F_1(\w)}}}{\n{U_\w(t)\,w(\w)}}=-\widetilde{\sigma}
\]
and
\[\lim_{t\to\infty}\frac{1}{t}\ln\n{U_\w(t)|_{F_1(\w)}}=\bar\lambda_2\]
for each $\w\in\widetilde\Omega_0$, that is, $\Phi$ admits a generalized exponential separation of type {\rm I}.
\end{itemize}
\end{teor}
\begin{nota}
Theorem~\ref{thm:exponential-separation-I} is new even in the case of generalized exponential separation of type I.  Indeed, under an additional assumption that $\bar\lambda_1 > -\infty$ it is stronger than \cite[Thm.~2.4]{MiShPart3}:  the latter requires that $\ln{\kappa}, \ln{\kappa^{*}} \in L_1\OFP$ and $\langle \mathbf{e}, \mathbf{e}^{*} \rangle > 0$.
\end{nota}
%%%%%%%%%%%%%%%%%%%%%%%%%%%%%%%%%%%%%%%%%%%%%%%%%%%%%%%%%%%%%%%%%%%%%%
%%%%%%%%%%%%%%%%%%%%%%%%%%%%%%%%%%%%%%%%%%%%%%%%%%%%%%%%%%%%%%%%%%%%%%
\section{Scalar linear random delay differential equations}\label{sec-delay}
This section is devoted to show the applications of the previous theory to random dynamical systems generated by scalar linear random delay differential equations of the form
\begin{equation}\label{5.linear}
z'(t)=a(\theta_t\w)\,z(t)+b(\theta_t\w)\,z(t-1)\,,\quad\w\in\Omega\,.
\end{equation}
\par\smallskip
Let $1<p<\infty$. We consider the separable Banach space $X=\R\times L_p([-1,0],\R)$ with the norm
\[\n{u}_X=|u_1|+\n{u_2}_p=|u_1|+\left(\int_{-1}^0 |u_2(s)|^p\,ds\right)^{\!\!1/p}\]
for any $u=(u_1,u_2)$ with $u_1\in\R$ and $u_2\in L_p([-1,0],\R)$. The positive cone \[X^+=\left\{u=(u_1,u_2)\in X\mid u_1 \geq 0 \text{ and } u_2(s)\geq 0 \text{ for Lebesgue-a.e.~} s\in[0,1]\right\}\] is normal and reproducing, and the dual  $X^*=\R\times L_q([-1,0],\R)$ with $1/q+1/p=1$ is also separable.\par
\smallskip
Now we introduce the assumptions on the coefficients of the family~\eqref{5.linear}:
\begin{itemize}
\item[(\textbf{S1})]
the $(\mathfrak{F}, \mathfrak{B}(\R))$-measurable functions $a$ and $b$ have the properties:
\begin{align*}
& \bigl[ \, \Omega \ni \omega \mapsto a(\w) \in \R \, \bigr] \in L_1\OFP, \text{ and}
\\
& \Bigl[ \, \Omega \ni \omega \mapsto \lnplus{\int_{0}^{1}\lvert b(\theta_{r}\w) \rvert^{q} \, dr} \in \R \, \Bigr] \in L_1\OFP.
\end{align*}
\item[(\textbf{S2})] $b(\w)\geq 0$ for each $\w\in\Omega$.
\end{itemize}
\begin{nota}
\label{rm:S1}
The following is sufficient for the fulfillment of the second condition in (S1):
\begin{equation*}
\label{eq:S1}
\bigl[ \, \Omega \ni \omega \mapsto b(\w) \in \R \, \bigr] \in L_q \OFP.
\end{equation*}
Indeed, since $\lvert b \rvert^{q} \in L_1\OFP$ and the measure $\PP$ is invariant, for any $t\in\R$
\begin{equation*}
\int_\Omega \lvert b(\theta_t\omega')\rvert^{q} \,d\PP(\omega') = \int_\Omega \lvert b(\omega')\rvert^{q} \,d\PP(\omega') \,
\end{equation*}
and an application of Fubini's theorem gives that the map
\begin{equation*}
\Bigl[ \, \Omega \ni \omega \mapsto \int_{0}^{1}\lvert b(\theta_{r}\w) \rvert^{q} \, dr \in \R \, \Bigr]
\end{equation*}
belongs to $L_1\OFP$, from which the required statement follows immediately.
\end{nota}
\par
\smallskip
In order to define the \mlsps\  we are going to deal with, for each $u=(u_1,u_2)\in X$ and $\w\in\Omega$ we consider the initial value problem
\begin{equation}\label{5.initial}
\left\{\begin{array}{l}
z'(t)=a(\theta_t\w)\,z(t)+b(\theta_t\w)\,z(t-1)\\[.1cm]
z(t)=u_2(t)\,,\quad t\in[-1,0)\,,\\[.1cm]
z(0)=u_1\,.
\end{array}\right.
\end{equation}
Its solution will be denoted by $z(t,\w,u)$.\par
\smallskip
Since $a\in L_1\OFP$ and the measure $\PP$ is invariant, for any $t\in\R$
\begin{equation*}
\int_\Omega a(\theta_t{\omega'})\,d\PP(\omega') = \int_\Omega a(\omega')\,d\PP(\omega') \,
\end{equation*}
and an application of Fubini's theorem gives that the map
\begin{equation}\label{5.aL1loc}
\bigl[\, \R\ni t\mapsto a(\theta_t\w)\in \R\,\bigr]\in L_{1,\text{loc}}(\R)
\end{equation}
for $\w\in\Omega_0\subset \Omega$,  invariant set of full measure. Then we can put the value of $a(\w)$  for $\w\in\Omega\setminus \Omega_0$ to be equal to zero to obtain~\eqref{5.aL1loc} for all $\w\in\Omega$.  Analogously, by changing the value of $b$ to zero in a set of null measure, the map
\begin{equation}\label{5.bLqloc}
\bigl[ \R\ni t\mapsto b(\theta_t\w)\in \R\bigr]\in L_{q,\text{loc}}(\R)\subset L_{1,\text{loc}}(\R) \,,
\end{equation}
for all $\w\in\Omega$.
Therefore, for a fixed $\w\in\Omega$ and $0\leq t\leq 1$ the system~\eqref{5.initial} of Carath\'{e}odory type has a unique solution, as shown by Coddington and Levinson~\cite[Theorem~1.1]{book:CL}, which can be written as
\begin{equation}
\begin{split}\label{5.0<t<1}
z(t,\w,u)& =\exp\biggl(\int_0^t \!\!a(\theta_r\w)\,dr\biggr)\left[u_1+\int_0^t \exp\biggl(-\int_0^s \!\!a(\theta_r\w)\,dr\biggr)\,b(\theta_s\w)\,u_2(s-1)\,ds\right]\\
 & = \exp\biggl(\int_0^t \!\! a(\theta_r\w)\,dr\biggr)\,u_1+ \int_0^t \exp\biggl(\int_s^t a(\theta_r\w)\,dr\biggr)\,b(\theta_s\w)\,u_2(s-1)\,ds\,,
\end{split}
\end{equation}
and, for $1\leq t\leq 2$ as
\begin{multline}\label{5.1<t<2}
z(t,\w,u)=\exp\biggl(\int_1^t a(\theta_r\w)\,dr\biggr)\biggl[z(1,\w,u)\\
 \quad +\int_1^t\exp\biggl(-\int_1^s a(\theta_r\w)\,dr\biggr)\,b(\theta_s\w)\,z(s-1,\w,u)\,ds\biggr].
\end{multline}
In a recursive way we obtain the formula for $z(t,\w,u)$ for any $t\in[-1,\infty)$.
\begin{nota}\label{5.positivefrom}
Assume that $u\in X^+$  and there is a $t_1\geq 0$ such that $z(t_1,\w,0)>0$. Then $z(t,\w,u)\geq \exp\bigl(\int_{t_1}^t a(\theta_r\w)\,dr\bigr)\,z(t_1,\w,u)>0$ for each $t\geq t_1$.
\end{nota}
\par\smallskip
Next we denote
\begin{equation}\label{5.defice}
c(\w) = \exp\biggr(\int_{0}^{1} \! \! \lvert a(\theta_r\w) \rvert \,dr\biggr) \; \text{ and }\; d(\w)=\biggl(\int_{-1}^0 \!b^q(\theta_{s+1}\w)\,ds\biggr)^{\!\!1/q}.
\end{equation}
\begin{lema}\label{5.desi[0,1]}
Under assumptions {\rm(S1)} and {\rm (S2)}, for each $\omega \in \Omega$, $0 \leq t \leq 1$ and $u \in X$ there holds
\begin{equation*}
\lvert z(t,\w,u) \rvert \leq c(\w)\,(1+d(\w))\,\n{u}_X\,.
\end{equation*}
\end{lema}
\begin{proof}  From~\eqref{5.0<t<1},~\eqref{5.bLqloc}, $b\geq 0$, $u_2\in L^p([-1,0],\R)$ and H\"{o}lder inequality, we deduce that if $0\leq t\leq 1$ and $u=(u_1,u_2)\in\R\times L_p([-1,0],\R)$
\begin{align*}
|z(t,\w,u)|&\leq c(\w)\biggl[\,|u_1|+\int_0^t b(\theta_s)\,|u_2(s-1)|\,ds\,\biggr]\\
           &= c(\w)\biggl[\,|u_1|+\int_{-1}^{t-1} b(\theta_{s+1})\,|u_2(s)|\,ds\,\biggr]\\
           &\leq c(\w)\biggl[\,|u_1|+\int_{-1}^0 b(\theta_{s+1})\,|u_2(s)|\,ds\,\biggr] \\
           & \leq c(\w)\biggl[\,|u_1|+ \biggl(\int_{-1}^0 b^q(\theta_{s+1}\w)\,ds\biggr)^{\!\!1/q}\n{u_2}_p\,\biggr]\\
           &\leq c(\w)\,(1+d(\w))\,\n{u}_X\,,
\end{align*}
as stated.
\end{proof}
Moreover, it can be checked that for each $t$ and $r\geq 0$
\begin{equation}\label{5.cocycle}
z(t+r,\w,u)=z(t,\theta_r\w,(z(r,\w,u),z_r(\w,u)))\,,
\end{equation}
where $z_t(\w,u)\colon [-1,0]\to \R$, $s\mapsto z(t+s,\w,u)$, which together with Lemma~\ref{5.desi[0,1]} show that $z_t(\w,u)\in L_p([-1,0],\R)$ for each $t\geq 0$
and we can define the linear~operator
\begin{equation}\label{5.defiskpd}
\begin{array}{lccc}
 U_\w(t)\colon & X &\longrightarrow & X \\[.1cm]
                        & u & \mapsto & (z(t,\w,u),z_t(\w,u))
\end{array}
\end{equation}
\begin{prop}\label{5.Uwbounded} Under assumptions {\rm(S1)} and {\rm (S2)}, $U_\w(t)$ satisfies~\eqref{eq-identity},~\eqref{eq-cocycle} and $U_\w(t)\in \mathcal{L}(X) $ for each $t\geq 0 $ and $\w\in\Omega$.
\end{prop}
\begin{proof} Relation~\eqref{eq-identity} is immediate and~\eqref{eq-cocycle} follows from~\eqref{5.cocycle}. Once that this cocycle property is shown, to prove that $U_\w(t)\in \mathcal{L}(X)$ for $t\geq 0$, it is enough to check that $U_\w(t)$ is a bounded operator for $t\in[0,1]$ and $\w\in\Omega$, which is a consequence of Lemma~\ref{5.desi[0,1]} because
\begin{align*}
\n{U_\w(t)\,u}_X&=|z(t,\w,u)|+\biggl(\int_{-1}^0 |z(t+s,\w,u)|^p\,ds\biggr)^{\!\!1/p}\leq c(\w)\,(1+d(\w))\,\n{u}_X\\
&\quad +\biggl(\int_{-1}^{-t} |u_2(t+s)|^p\,ds\biggr)^{\!\!1/p}+\biggl(\int_{-t}^0 |z(t+s,\w,u)|^p\,ds\biggr)^{\!\!1/p}\\
&\leq 3\,c(\w)\,(1+d(\w))\,\n{u}_X\,,
\end{align*}
that is, $\n{U_\w(t)}\leq 3\,c(\w)\,(1+d(\w))$ for $t\in[0,1]$, which finishes the proof.
\end{proof}
In order to show that $\Phi = \allowbreak ((U_\w(t))_{\w \in \Omega, t \in \R^{+}}, \allowbreak (\theta_t)_{t\in\R})$ is a \mlsps\, we start with the following auxiliary lemma.
\begin{lema}\label{5.measurable}
Under  {\rm(S1)} and {\rm (S2)}, for each $u\in X$ and $t>0$ the mapping
\[ \bigl[\,\Omega\ni\w\mapsto U_\w(t)\,u\in X\,\bigr] \text{ is } (\mathfrak{F},
\mathfrak{B}(X))\text{-measurable}\,.\]
\begin{proof} It follows from~\eqref{eq-cocycle} that it suffices to prove the result for $t\in(0,1]$ only.  Since $X$ is separable, from Pettis' Theorem (see~Hille and Phillips~\cite[Theorem 3.5.3 and Corollary 2 on pp. 72--73]{HiPhi}) the weak and strong measurability notions are equivalent and therefore,  it is enough to check that for each $u^*\in X^*$ the mapping
\begin{equation}\label{5.u*measur}
\bigl[\,\Omega\ni\w\mapsto u^*(U_\w(t)\,u)\in \R\,\bigr] \text{ is } (\mathfrak{F},
\mathfrak{B}(\R))\text{-measurable}\,.
\end{equation}
Fixing $u^*=(u_1^*,u_2^*)\in\R\times L_q([-1,0],\R)$, $u=(u_1,u_2)\in \R\times L_p([-1,0],\R)$ and $t\in(0,1]$,  we have
\begin{equation}\label{5.u*(U)}
u^*(U_w(t)\,u)= z(t,\w,u)\,u_1^*+ \int_{-1}^0 z(t+s,w,u)\,u_2^*(s)\,ds\,.
\end{equation}
The measurability of the map
\[
\begin{array}{ccccc} \Omega\times[0,t]& \to &\Omega &\to&\R\\
                        (\w,r) & \mapsto & \theta_r\w &\mapsto & a(\theta_r\w)
\end{array}
\]
 and an application of Fubini's theorem show that $\bigl[\,\Omega\ni\w \mapsto \int_0^t a(\theta_r\w)\,dr\in \R\,\bigr]$ is $(\mathfrak{F},\mathfrak{B}(\R))$\nbd-measurable. Therefore,
 \begin{equation}\label{5.zmeasu}
 \bigl[\,\Omega\ni\w \mapsto \exp\bigl({\textstyle\int_0^t a(\theta_r\w)\,dr}\big)\,u_1\in\R \,\bigr]  \text{ is } (\mathfrak{F},\mathfrak{B}(\R))\text{-measurable}\,.
 \end{equation}
 Since $\int_s^t a(\theta_r\w)\,dr=\int_0^t\chi_{[0,t]}(t-s)\,a(\theta_{r+s}\w)\,dr$, again  the measurability of the mapping
\[
\begin{array}{ccc} \Omega\times[0,t]\times [0,t]& \to &\R\\
                        (\w,r,s) & \mapsto & \chi_{[0,t]}(t-s)\,a(\theta_{r+s}\w)
\end{array}
\]
and Fubini's theorem prove that the maps $\bigl[\,\Omega\times[0,t]\ni (\w,s)\mapsto \int_s^t a(\theta_r\w)\,dr\in \R\,\bigr]$  and  $\bigl[\,\Omega\times[0,t]\ni (\w,s)\mapsto \exp\bigl(\int_s^t a(\theta_r\w)\,dr\bigr)\,b(\theta_s\w)\,u_2(s-1)\in \R\,\bigr]$ are measurable. From this, as before, we deduce that
\begin{equation*}
 \bigl[\,\Omega\ni\w \mapsto \textstyle{\int_0^t \exp\bigl({\int_s^t \!\!a(\theta_r\w)\,dr\bigr)}\,b(\theta_s\w)\,u_2(s-1)\,ds\in\R}\, \bigr]  \text{ is } (\mathfrak{F},\mathfrak{B}(\R))\text{-measurable}\,,
 \end{equation*}
 which together with~\eqref{5.zmeasu} and formula~\eqref{5.0<t<1}
prove that $\bigl[\,\Omega\ni\w \mapsto z(t,\w,u)\in \R\,\bigr]$ is $(\mathfrak{F},
\mathfrak{B}(\R))$\nbd-measurable.
Finally from this fact, the formula
\begin{align*}
\int_{-1}^0 z(t+s,\w,u)\,u_2^*(s)\,ds& =\int_{-1}^{-t} u_2(t+s)\,ds+\int_{-t}^0 z(t+s,\w,u)\,u_2^*(s)\,ds\\
&= \int_{-1}^{-t} u_2(t+s)\,ds+\int_{0}^t z(s,\w,u)\,u_2^*(s-1)\,ds\,,
\end{align*}
a similar argument, and~\eqref{5.u*(U)} we conclude that~\eqref{5.u*measur} holds, as claimed.
\end{proof}
\end{lema}
In view of the above, we call $\Phi = \allowbreak ((U_\w(t))_{\w \in \Omega, t \in \R^{+}}, \allowbreak (\theta_t)_{t\in\R})$ as defined by~\eqref{5.defiskpd} the \mlsps\ {\em generated by~\eqref{5.linear}}.
\par
\smallskip
We finish this section showing that the skew-product semidynamical system, generated by the family of scalar linear random delay differential equations of the form~\eqref{5.linear}, satisfies all the requirements for the existence of a generalized exponential separation.  Notice that according to Remark~\ref{timeT} we can take time $T = 2$ instead of 1 to check conditions (A1) and (A3).
\par
\smallskip
Before proceeding we formulate and prove the following auxiliary
\begin{lema}
\label{lm:auxiliary2}
Let $x_1, \dots, x_n > 0$.  Then
\begin{equation*}
\lnplus{\Bigl(\sum_{i=1}^{n} x_i\Bigr)} \le \sum_{i=1}^{n} \lnplus{x_i} + \ln{n}.
\end{equation*}
\end{lema}
\begin{proof}
Applying the Jensen inequality to the convex function $f(x) = x \ln{x}$ we obtain
\begin{equation*}
\Bigl( \sum_{i=1}^{n} x_i  \Bigr) \, \ln{\Bigl( \frac{1}{n}} \sum_{i=1}^{n} x_i \Bigr) \le \sum_{i=1}^{n} x_i \ln{x_i},
\end{equation*}
which gives
\begin{equation*}
\ln{\Bigl(\sum_{i=1}^{n} x_i\Bigr)} \le \sum_{i=1}^{n} \frac{x_i}{\sum_{j=1}^{n} x_j} \ln{x_i} + \ln{n}.
\end{equation*}
For $i$ such that $\ln{x_i} \le 0$ we have
\begin{equation*}
\frac{x_i}{\sum_{j=1}^{n} x_j} \ln{x_i} \le 0 = \lnplus{x_i},
\end{equation*}
whereas for $i$ such that $\ln{x_i} > 0$ we have
\begin{equation*}
\frac{x_i}{\sum_{j=1}^{n} x_j} \ln{x_i} \le \ln{x_i} = \lnplus{x_i},
\end{equation*}
consequently
\begin{equation*}
\ln{\Bigl(\sum_{i=1}^{n} x_i\Bigr)} \le \sum_{i=1}^{n} \lnplus{x_i} + \ln{n}.
\end{equation*}
As the right-hand side of the above inequality is nonnegative, we obtain the desired result.
\end{proof}
\begin{prop}
Under  {\rm(S1)} and {\rm (S2)}, $\Phi = \allowbreak ((U_\w(t))_{\w \in \Omega, t \in \R^{+}}, \allowbreak (\theta_t)_{t\in\R})$ is a \mlsps\ satisfying properties {\rm(A1)}, {\rm (A2)} and {\rm (A3)} for time $T=2$. Moreover, the generalized Lyapunov exponent satisfies $\bar\lambda_1\geq \int_\Omega a \,d\PP$.
\end{prop}
\begin{proof} Since for $\w\in\Omega$ and $u\in X$ fixed the mapping $\bigl[\,\R^+\ni t\mapsto U_\w(t)\,u\in X\,\bigr]$ is easily seen to be continuous, the fact that the mapping
\[
\bigl[\,\R^+\times \Omega\times X\ni(t,\w,u) \mapsto U_\w(t)\,u\in X\,\bigr]\]
 is  $(\mathfrak{B}(\R^+)\otimes\mathfrak{F}\otimes\mathfrak{B}(X),
\mathfrak{B}(X))$-measurable follows from Proposition~\ref{5.Uwbounded}, Lemma~\ref{5.measurable} and Aliprantis and Border~\cite[Lemma 4.51 on pp.~153]{AliB}. The rest of the properties have been already checked, so that $\Phi$ is a \mlsps, as claimed. \par
\smallskip
 Concerning the first part of (A1) with $T=2$  notice that
 \[\sup\limits_{0 \le s \le 2} {\lnplus{\n{U_{\omega}(s)}}}\leq \sup\limits_{0 \le s \le 1} {\lnplus{\n{U_{\omega}(s)}}}+ \sup\limits_{1 \le s \le 2} {\lnplus{\n{U_{\omega}(s)}}}\]
 and from cocycle property~\eqref{eq-cocycle}
 \[
\sup\limits_{1 \le s \le 2} \lnplus{\n{U_{\omega}(s)}}=\sup\limits_{0 \le s \le 1} {\lnplus{\n{U_{\omega}(1+s)}}}\leq \sup\limits_{0 \le s \le 1} \bigl(\lnplus\n{U_{\theta_s\omega}(1)}+ \lnplus \n{U_{\w}(s)}\bigr).
\]
Therefore
\[\sup\limits_{0 \le s \le 2} {\lnplus{\n{U_{\omega}(s)}}}\leq 2\,\sup\limits_{0 \le s \le 1} {\lnplus{\n{U_{\omega}(s)}}}+ \sup\limits_{0 \le s \le 1} {\lnplus{\n{U_{\theta_s\omega}(1)}}}\]
and we have to check that both terms belong to $L_1\OFP$.
\par
\smallskip
As shown in Proposition~\ref{5.Uwbounded}, $\n{U_\w(t)}\leq 3\,c(\w)\,(1+d(\w))$ for each $\w\in\Omega$ and $t\in[0,1]$, where $c(\w)$ and $d(\w)$ are defined in~\eqref{5.defice}.
From $1 \leq c(\w) = \exp\bigl({\int_{0}^1 \lvert a(\theta_r\w) \rvert \,dr}\bigr) $  we deduce that  $\lnplus c(\w) = \ln c(w) = \int_0^1 \lvert a(\theta_r\w) \rvert \, dr$ belongs to $L_1\OFP$ because of (S1), Fubini's theorem and the invariance of $\PP$. Analogously, $\lnplus d(\w)$ belongs to $L_1\OFP$ because of (S1), and therefore, with the help of Lemma~\ref{lm:auxiliary2},
\begin{align*}
 \sup\limits_{0 \le s \le 1} {\lnplus{\n{U_{\omega}(s)}}}
&\le \ln(3\,c(\w)) + \ln(1+d(\w))
\\
& \le   \ln (3\,c(\w)) + \lnplus{d(\w)} + \ln{2}
= \lnplus{c(\w)} + \lnplus{d(\w)} + \ln{6}
\end{align*}
also belongs to $L_1\OFP$. For the second term, $\n{U_{\theta_s\w}(1)}\leq 3\,c(\theta_s\w)\,(1+d(\theta_s\w))$ for $0\leq s \leq 1$\,,
\begin{align*}
\lnplus c(\theta_s\w)& = \ln c(\theta_{s}(\w)) = \int_0^1 \lvert a(\theta_{r+s}\w) \rvert \, dr \leq \int_0^2 \lvert a(\theta_r\w) \rvert \,dr \, , \\
\lnplus (1+d(\theta_s\w)) & \le \lnplus{d(\theta_s\w)} + \ln{2} = \frac{1}{q} \lnplus{\int_{0}^2 b^q(\theta_{r}\w)\,dr} + \ln{2}
\\
& \le \frac{1}{q} \biggl( \lnplus{\int_{0}^1 b^q(\theta_{r}\w)\,dr} + \lnplus{\int_{0}^1 b^q(\theta_{1 + r}\w)\,dr} + \ln{2} \biggr) + \ln{2}
\\
& = \lnplus{d(\w)} + \lnplus{d(\theta_1\w)} + \frac{1 + q}{q}\, \ln{2}\,,
\end{align*}
and hence, an analogous argument using (S1), Fubini's theorem and the invariance of the measure $\PP$ proves that
\[\bigl[ \, \Omega \ni \omega \mapsto \sup\limits_{0 \le s \le 1}\,
{\lnplus{\n{U_{\theta_{s}\omega}(1)}}} \in [0,\infty) \, \bigr] \in L_1\OFP\]
and the first assertion of (A1) holds. We omit the second part of (A1) because it is analogous.
It is immediate to check that (A2) follows from (S2).
\par
\smallskip
We will finish by verifying that (A3) holds for time $T=2$. We consider the vector $\be=(1/2,u_0)\in X^+$ with $u_0(s)=1/2$ for each $s\in[-1,0]$.  We have $\n{\be}_X=1$ and it is immediate, via Remark~\ref{5.positivefrom}, to check that $z(t,\w,\be)>0$ for each $t\geq 0$, which in particular implies that $U_\w(2)\,\be=(z(2,\w,\be),z_2(\w,\be))\neq 0$.\par\smallskip
Let $u\in X^+$ such that $U_\w(2)\,u=(z(2,\w,u),z_2(\w,u))\neq 0$. We claim that $z(1,\w,u)>0$. Assume on the contrary that $z(1,\w,u)=0$. From Remark~\ref{5.positivefrom} we also deduce that
 $z(t,\w,0)=0$ for each $0\leq t\leq 1$ and hence $U_\w(1)\,u=0$, in contradiction with $U_{\theta_1\w}(U_\w(1)\,u)=U_\w(2)\,u\neq 0$.
From~\eqref{5.1<t<2} we deduce that $z(2+s,\w,u)\geq \exp\bigl({\int_1^{2+s}a(\theta_r\w)\,dr}\bigr)\,z(1,\w,u)$, and hence
\begin{equation}\label{5.leftineq}
2\,\inf_{1\leq t\leq 2}\exp\biggl({\int_1^t a(\theta_s\w)\,dr\biggr)}\,z(1,\w,u)\,\be\leq U_\w(2)\,u\,.
\end{equation}
A straightforward computation shows that
\[y'(t,\w,u)=\exp\bigg(-\int_{t-1}^t\!\!a(\theta_r\w)\,dr\bigg)\,b(\theta_t\,\w)\,y(t-1,\w,u)\]
where $z(t,\w,u)=\exp\bigl({\int_0^ta(\theta_r\w)\,dr}\bigr)\,y(t,\w,u)$.  Since $y'\geq 0$,  we obtain for $1\leq s\leq 2$
\[z(s-1,\w,u)\leq \exp\biggl(\int_0^{s-1}\!\!a(\theta_r\w)\,dr\biggr)\,y(1,\w,u)=\exp\biggr(-\int_{s-1}^1 \!\!a(\theta_r\w)\,dr\biggr)\,z(1,\w,u)\,.\]
Consequently, again \eqref{5.1<t<2} yields
\begin{align*}
z(t,\w,u)& \leq \exp\biggl({\int_{1}^t \!\!a(\theta_r\w)\,dr}\biggr)z(1,\w,u)\biggl[ 1+\int_1^t \exp\biggl({-\int_{s-1}^s \!\!a(\theta_r\w)\,dr}\biggr)\,b(\theta_s\w)\,ds\biggr]\\
&\leq \sup_{1\leq t\leq 2}\exp\biggl(\int_1^t \!\!a(\theta_{r}\w)\,dr\!\!\biggr)\,z(1,\w,u)\\ & \hspace{4.5cm}\cdot\biggl[ 1 + \int_1^2 \exp\biggl({-\int_{s-1}^s\!\! a(\theta_r\w)\,dr}\biggr)\,b(\theta_s\w)\,ds\biggr]
\end{align*}
for each $1\leq t\leq 2$. Moreover,
\[\frac{{\displaystyle \sup_{1\leq t\leq 2}} \exp\bigl({\int_1^t a(\theta_{r}\w)\,dr}\bigr)}{{\displaystyle \inf_{1\leq t\leq 2}} \exp\bigl({\int_1^t a(\theta_{r}\w) \, dr}\bigr)} \leq c(\theta_1\w)\,,\]
and therefore,
\begin{align*}
z(t,\w,u)& \leq \inf_{1\leq t\leq 2}\exp\biggl({\int_1^t \!\!a(\theta_{r}\w)\,dr}\biggr)\, c(\theta_1\w)\,z(1,\w,u)\\
&\hspace{4cm}\cdot \biggl[\, 1+\int_1^2 \exp\biggl({-\int_{s-1}^s \!\!a(\theta_r\w)\,dr}\biggr)\,b(\theta_s\w)\,ds\,\biggr]
\end{align*}
Finally, denoting by
\begin{align*}
\beta(\w,u)&=2\,\inf_{1\leq t\leq 2} \exp\bigg({\int_1^t \!\!a(\theta_{r}\w)\,dr}\biggr)\,z(1,\w,u)>0\,,\\ \varkappa(\w)& = c(\theta_1\w)\biggl[\, 1+\int_1^2 \exp\biggl({-\int_{s-1}^s \!\!\!a(\theta_r\w)\,dr}\biggr)\,b(\theta_s\w)\,ds\,\biggr]\,,
\end{align*}
and~\eqref{5.leftineq} we obtain
\[\beta(\w,u)\,\be\leq U_\w(2)\,u\leq \varkappa(\w)\,\beta(\w,u)\,\be\,.\]
In order to finish the proof we have to check that $\lnplus\ln\varkappa\in L_1((\Omega,\mathcal{F},\PP))$. However, notice that  $\lnplus\ln\varkappa(\w) \leq \ln\varkappa(\w)$ because $\varkappa (\w)\geq 1$ and thus, it suffices to prove that
$\ln \varkappa\in L_1((\Omega,\mathcal{F},\PP))$.
In addition, for $1\leq s\leq 2$
\[\exp\biggl({-\int_{s-1}^s \!\!a(\theta_r\w)\,dr}\biggr) \leq \exp\biggl({\int_{0}^2  \!\! \lvert a(\theta_r\w) \rvert \, dr} \biggr) = c(\w)\,c(\theta_1\w) \, , \]
and we deduce that
\begin{equation*}
\varkappa(\w) \leq c(\theta_1\w)^2\,c(\w)\biggl[ 1+\int_1^2 b(\theta_s\w)\,ds\biggr]\,,
\end{equation*}
consequently, with the help of Lemma~\ref{lm:auxiliary2} and the H\"older inequality,
\begin{align*}
\lnplus{\varkappa(\w)} & \le 2 \lnplus{c(\theta_1\w)} + \lnplus{c(\w)} + \lnplus{\biggl( 1 + \int_1^2 b(\theta_s\w) \, ds \biggr)}
\\
& \le 2 \lnplus{c(\theta_1\w)} + \lnplus{c(\w)} + \lnplus{\int_1^2 b(\theta_s\w) \, ds} + \ln{2}
\\
& \le 2 \lnplus{c(\theta_1\w)} + \lnplus{c(\w)} + \lnplus{\int_0^1 b(\theta_{1 + s}\w) \, ds} + \ln{2} \,,
\end{align*}
from which, as before, together with (S1), Fubini's theorem and the invariance of $\PP$, we conclude that $\ln \varkappa\in L_1((\Omega,\mathcal{F},\PP))$  and (A3) holds, as stated.
The inequality $\bar\lambda_1 \geq \int_\Omega a \, d\PP$ follows from Remark~\ref{5.positivefrom} and Birkhoff ergodic theorem.
\end{proof}
To sum up, we have proved the following.
\begin{teor}
\label{thm:exp-sep-appl}
Assume \textup{(S1)} and \textup{(S2)}.  Then the \mlsps\ $\Phi = \allowbreak ((U_\w(t))_{\w \in \Omega, t \in \R^{+}}, \allowbreak (\theta_t)_{t\in\R})$ generated by~\eqref{5.linear} admits a generalized exponential separation of type II, with $\tilde{\lambda}_1 > - \infty$.
\end{teor}
%%%%%%%%%%%%%%%%%%%%%%%%%%%%%%%%%%%%%%%%%%%%%%%%%%%%%%%%%%%%%%%%%%%%%%
%%%%%%%%%%%%%%%%%%%%%%%%%%%%%%%%%%%%%%%%%%%%%%%%%%%%%%%%%%%%%%%%%%%%%%

\end{document}